\documentclass[preprint,11pt]{elsarticle}
\usepackage[T1]{fontenc}
\usepackage[latin9]{inputenc}
\usepackage{multirow}
\usepackage{array} 
\usepackage{amsmath, amssymb}
\usepackage{empheq}

\usepackage{amsthm} 

 \newtheorem{theorem}{Theorem}[section]

\newtheorem{proposition}[theorem]{Proposition}
\newtheorem{remark}[theorem]{Remark}
\newtheorem{definition}[theorem]{Definition}
\newtheorem{corollary}[theorem]{Corollary}
 
\usepackage{color}
\usepackage{graphicx,color}
\usepackage{amsmath, amssymb, graphics}

\begin{document}
\begin{frontmatter}
\title{ Rotational surfaces of constant astigmatism  in space forms }
  \author[label1]{Rafael L\'opez}
 \ead{rcamino@ugr.es}
 \address[label1]{ Departamento de Geometr\'{\i}a y Topolog\'{\i}a\\ Instituto de Matem\'aticas (IEMath-GR)\\
 Universidad de Granada\\
 18071 Granada, Spain}
 \author[label2]{\'Alvaro P\'ampano}
 \ead{alvaro.pampano@ehu.eus}
 \address[label2]{Departament of  Mathematics\\ Faculty of Science and Technology\\
University of the Basque Country\\
48940 Bilbao, Spain}

\begin{abstract}
A surface in a Riemannian space   is called of constant astigmatism if the difference between the principal radii of curvatures at each point is a constant function. In this paper we give a   classification of all rotational surfaces of constant astigmatism in      space forms. We also prove that the generating curves of such surfaces are critical points   of a variational problem for a curvature energy.  Using the description of these curves, we locally construct all rotational surfaces of constant astigmatism as the associated binormal evolution surfaces from the generating curves.
\end{abstract}
\end{frontmatter}

\section{Introduction}

In the    Euclidean space $\mathbb{R}^3$, a surface of constant astigmatism is a surface where the difference between the principal radii of curvatures $\rho_2-\rho_1$ at each point is a constant function. From the physical viewpoint, this difference $\rho_2-\rho_1$    measures the amplitude of astigmatism and thus, a surface where this difference is constant  has constant amplitude of astigmatism in the normal directions (\cite{gi,st}). The interest on these   curves lies on the property that their lifts to the space $\mathbb{R}^2\times\mathbb{S}^1$ endowed with a suitable sub-Riemannian structure are believed to be used by the brain to complete contours of pictures that are missing to the eye vision.   See \cite{p} and references therein for more details. Definitively, the   visual curve completion problem and  surfaces of constant astigmatism are geometrically related.

Surfaces of constant astigmatism were studied in early works of Bianchi and Ribaucour proving   that their focal surfaces have constant negative Gaussian curvature \cite{bi1,bi2,ku,ri}. For this reason, the Gauss equation  of these surfaces is related with the   sine-Gordon equation, which is known to be integrable in the sense of the soliton theory. Recently there has been an increasing interest in studying the Gauss equation of these surfaces using the theory of integrable systems: \cite{bm,hl0,hl,hm0,hm1,mp,pz}. For our purposes, we need to extend the notion of constant astigmatism surfaces in  any  $3$-dimensional Riemannian space.

\begin{definition} Let $M^3$ be a  $3$-dimensional Riemannian space and let $S$ be an oriented surface in $M^3$  with non-zero principal curvatures $\kappa_1$ and $\kappa_2$. We say that $S$ is a surface of  constant astigmatism if  there is a constant $c\in\mathbb{R}$ such that 
\begin{equation} \frac{1}{\kappa_1(p)}-\frac{1}{\kappa_2(p)}=c\label{relation}
\end{equation}
holds for any $p\in S$. 
\end{definition}

Notice that if   $c=0$ in   \eqref{relation},  then   $S$   is  totally umbilical, hence we will discard this trivial case and   we will  assume from now on that $c\not=0$.  The constant astigmatism equation \eqref{relation} can be also viewed as a relation between the principal curvatures and thus $S$ is a Weingarten surface. The class of Weingarten surfaces is a topic of great interest for geometers, especially after the works of Hopf, Chern and Hartman among others in the 1950s.  In the case of our study, equation \eqref{relation} establishes a linear relation between the principal radii of curvature.  A similar setting is to  consider a linear relation between the principal curvatures $\kappa_1$ and $\kappa_2$. In \cite{lp}, the authors have given   a complete description of all rotational surfaces in Euclidean space satisfying the relation $a\kappa_1+b\kappa_2=c$, where $a,b,c\in\mathbb{R}$.
 
In order to construct examples of surfaces of constant astigmatism, it is natural to impose some symmetry properties on the surface. A first remarkable class of surfaces are the rotational surfaces (or surfaces of revolution). Rotational surfaces  of constant astigmatism  in $\mathbb{R}^3$ were described by Lilienthal  proving that the profile curves   are the involutes of the tractrix curve that generates the pseudosphere. A picture of the generating curves of the Lilienthal's surfaces appears in \cite[Fig. 1]{bm} and in    figure \ref{curvasr21} of the present paper. In general, a surface of constant astigmatism is the involute of a surface of constant negative Gaussian curvature (\cite{bm}). On the other hand,  it has been proved in \cite{p}  that tractrices of the pseudosphere are the critical curves of a total curvature type variational problem. More generally,   the same paper proves that the profile curves of any rotational surface with constant negative Gaussian curvature are      critical curves   of the   total curvature type  energy
\begin{equation}\label{fff}
\mathcal{F}(\gamma)=\int_\gamma \sqrt{\kappa^2+a^2}
\end{equation}
for some   constant $a\not=0$. 

The aim of this paper is twofold. First, we extend the surfaces of constant astigmatism in   space forms, that is, including  the  sphere and the hyperbolic  space.  We will classify all rotational surfaces of constant astigmatism  in  space forms, giving a full description of them (in the hyperbolic space  we only consider rotational surfaces of spherical type). The second objective   is to  prove that the generating curves are solutions of a variational problem associated to a curvature energy involving the curvature $\kappa$ and the constant $c$ in \eqref{relation}. This energy is measured in a  suitable space of curves that deform the initial curve.

This paper is organized as follows.  In Section \ref{sec2} we define an energy functional $\mathbf{\Theta}_\mu$ for curves in $2$-space forms and we calculate the formula of the first variation of $\mathbf{\Theta}_\mu$. In Section \ref{sec3} we prove that the critical curves of $\mathbf{\Theta}_\mu$ are the generating curves of rotational surfaces of constant astigmatism  in  space forms. In Section \ref{sec4} we will prove the converse process   by evolving critical curves of $\mathbf{\Theta}_\mu$ under their associated binormal flow giving  a way of constructing all    rotational surfaces of constant astigmatism in  space forms. Once we have characterized the generating curves, we proceed to its classifications. Firstly, in Section \ref{sec5} we give some geometric properties of the critical curves of $\mathbf{\Theta}_\mu$ and  we associate to the critical curve a system of ODE  where we analyze the phase plane and its singular points.   Finally,   Section \ref{sec6} is devoted to classify and describe all extremal curves of $\mathbf{\Theta}_\mu$ in each space form: see Theorems \ref{curvesR2}, \ref{curvesS21}, \ref{curvesS22} and \ref{curvesH2} for precise statements. In particular, in Euclidean space we revisit the Lilienthal's surfaces obtaining a full description of all cases in terms of an integration constant. 

\section{The curvature energy problem}\label{sec2}

Let ${\mathbb M}^n(\rho)$ denote a   $n$-dimensional Riemannian space form of constant sectional curvature $\rho$, that is, the Euclidean space ($\rho=0$), the round sphere $\mathbb{S}^n(\rho)$ ($\rho>0$) and the hyperbolic space $\mathbb{H}^n(\rho)$ ($\rho<0$). We   denote by $\widetilde{\nabla}$ the Levi-Civita connection on ${\mathbb M}^n(\rho)$.  For $n=3$, let $\gamma(s)$ be a curve   in ${\mathbb M}^3(\rho)$ parametrized by  arc-length  and let $T(s)=\gamma'(s)$ represent the \emph{unit tangent vector field} of $\gamma$. If $\widetilde{\nabla}_T T(s)$ vanishes, then $\gamma$ is a geodesic and  we say that the curvature of $\gamma$ is identically zero. If $\gamma(s)$ is not a geodesic,   then $\gamma(s)$ is a \textit{Frenet curve} of
rank $2$ or $3$  and the  standard \textit{Frenet frame} along
$\gamma(s)$ is defined by $\{T,N, B\}$, where $N$ and $B$ are the
\emph{unit normal vector field} and \emph{unit binormal vector field},
respectively. The
 Frenet equations are 
\begin{equation}\label{frenet}
\begin{split}
\widetilde{\nabla}_T T(s)&=\kappa(s) N (s)\\
\widetilde{\nabla}_T N(s)&=-\kappa(s) T(s) +\tau(s) B(s)\\
\widetilde{\nabla}_T B(s)&= -\tau (s)N (s),
\end{split}
\end{equation}
where $\kappa$ and $\tau$ are  the \textit{curvature} and  the \textit{torsion}  of $\gamma$, respectively. Notice that  if the rank
of $\gamma$ is $2$, which occurs when $\tau =0$, the curve can be assumed to lie in a totally geodesic surface which is identified with a $2$-space form ${\mathbb M}^2(\rho)$. The curves whose torsion vanishes identically are called \emph{planar curves}. Conversely,  a curve in ${\mathbb M}^2(\rho)$ can be viewed  assuming that ${\mathbb M}^2(\rho)$ is a totally geodesic surface of ${\mathbb M}^3(\rho)$ with $\tau=0$ and the binormal vector 
$B=T\times N$ is   well defined and constant.

  Let  $\gamma:[0,L]\rightarrow {\mathbb M}^2(\rho)$ be a  curve parametrized   by the arc-length parameter $s$. For any real constant $\mu\in\mathbb{R}$,  called the {\it energy index}, define the   \emph{curvature
energy functional}
\begin{equation}
\mathbf{\Theta}_\mu(\gamma )=\int_{\gamma} \kappa\, {\rm e}^{\mu/\kappa}=
\int_{0}^L  \kappa(s){\rm e}^{\mu/\kappa(s)} ds.\label{energy}
\end{equation}
 Notice that if $\mu=0$ the curvature energy $\mathbf{\Theta}_0$  is just the  functional $\mathcal{F}$ in \eqref{fff} for $a=0$ and it represents the total curvature of $\gamma$. From now on, we assume  $\mu\neq 0$. We consider  $\mathbf{\Theta}_\mu$ acting on the   space $\Omega_{p_0p_1}$ of (non-geodesic) planar curves     joining two given points $p_0, p_1\in {\mathbb M}^2(\rho)$:
$$
\Omega_{p_0p_1}=\{\iota:[0,1]\rightarrow {\mathbb M}^2(\rho): \iota (i)=p_i, i\in\{0,1\}, \frac{d\iota}{dt}(t)\neq 0,
\forall t\in [0,1] , \kappa\neq 0 \}.$$

Let $\Gamma$ be a variation of $\gamma$, that is, 
$\Gamma=\Gamma(s,t):[0,L]\times
(-\varepsilon,\varepsilon)\rightarrow {\mathbb M}^2(\rho)$ is a smooth map with
$\Gamma(s,0)=\gamma(s)$ and let  $W=W(s)=\frac{\partial\Gamma}{\partial t}(s,0)$ be the variational vector field
along     $\gamma$. We have defined a family of curves $\{\Gamma(\text{-},t):t\in (\varepsilon,\varepsilon)\}$ and, as usually, we   write
$V=V(s,t)=\frac{\partial\Gamma}{\partial s}(s,t)$,
$W=W(s,t)=\frac{\partial\Gamma}{\partial t}(s,t)$, $v=v(s,t)=\vert V(s,t)\vert$,  and so on. For the next computations, we follow \cite{AGP} and \cite{Langer-Singer}. By using   the Frenet equations \eqref{frenet},    the variations of $v$ and $\kappa$  in the direction of  $W$ are
\begin{eqnarray*}
W(v) & = & v\langle\widetilde{\nabla}_T W,T\rangle,\\
W(\kappa) & = & \langle\widetilde{\nabla}^2_T W,N\rangle-2\kappa\langle\widetilde{\nabla}_T W,T\rangle +\rho\langle W,N\rangle.  
\end{eqnarray*}
  After a standard
computation involving integration by parts and the above expressions of $W(v)$ and $W(\kappa)$,     the \emph{first variation formula} of $\mathbf{\Theta}_\mu$ is
\begin{equation}
{\frac{d}{dt}}{\Big|}_{t=0}\mathbf{\Theta}_\mu(\Gamma(\text{-},t))=\
 \int_{0}^{L}\langle\mathcal{E}(\gamma),W\rangle ds +\mathcal{B}
 \left[ W,\gamma \right] _{0}^{L},  \label{1fv} \nonumber
\end{equation}
where   
\begin{eqnarray*} 
&&\mathcal{E}(\gamma)  =\widetilde{\nabla}_T \mathcal{J}-R(T,\mathcal{K})T=\widetilde{\nabla}_T \mathcal{J}+\rho\mathcal{K},  \\ 
&&\mathcal{B}\left[ W,\gamma \right] _{0}^{L} = \left[\langle\mathcal{K},\widetilde{\nabla}_T W\rangle - \langle \mathcal{J},W\rangle\right] _{0}^{L}, 
\end{eqnarray*}
and
\begin{eqnarray}
\mathcal{K}(\gamma) &=&\left(1-\frac{\mu}{\kappa}\right){\rm e}^{\mu/\kappa}N,  \label{kfield} \\
\mathcal{J}(\gamma)&
=&\widetilde{\nabla}_T \mathcal{K}+\left(\kappa-2\mu\right){\rm e}^{\mu/\kappa} T. \label{ejk}
\end{eqnarray}

A curve $\gamma\subset\Omega_{p_0p_1}$  is called  a \emph{critical curve} or \emph{extremal curve}  if $\mathcal{E}(\gamma)=0$ for any variation $\Gamma$ of $\gamma$. Notice that this is an abuse of notation because proper criticality depends on the boundary conditions, as it is clear by the boundary term $\mathcal{B}
 \left[ W,\gamma \right] _{0}^{L}$ in the first variation formula. However, under suitable boundary conditions, curves satisfying $\mathcal{E}(\gamma)=0$ are going to be proper critical curves. For our purposes, we just need to consider curves satisfying $\mathcal{E}(\gamma)=0$ and   for simplicity, we will use the name critical curve (or extremal curve) to denote a curve $\gamma\subset\Omega_{p_0 p_1}$ satisfying  $\mathcal{E}(\gamma)=0$ for any variation $\Gamma$ of $\gamma$. 

Using the first and second Frenet equations in  \eqref{frenet}, a  straightforward computation shows that $\mathcal{E}(\gamma)$ has no component in $T$. Then, identity $\mathcal{E}(\gamma)=0$ is equivalent to the vanishing of the component in $N$, obtaining the Euler-Lagrange equation for $\mathbf{\Theta}_\mu$.

\begin{proposition} Let $\gamma=\gamma(s)$ be a non-geodesic curve in ${\mathbb M}^2(\rho)$. Then   $\gamma$ is a critical curve of the curvature energy $\mathbf{\Theta}_\mu$ if and only if
\begin{equation}
\frac{d^2}{ds^2}\left(\left(1-\frac{\mu}{\kappa}\right){\rm e}^{\mu/\kappa}\right)+\left(\rho\left(1-\frac{\mu}{\kappa}\right)-\mu\kappa\right){\rm e}^{\mu/\kappa}=0. \label{EL}
\end{equation}
\end{proposition}

We study the Euler-Lagrange equation \eqref{EL} in the particular case that the curvature of  $\gamma$  is constant.  

\begin{corollary}\label{ccc} The only non-geodesic critical curves of $\mathbf{\Theta}_\mu$ with constant curvature $\kappa_0$ in ${\mathbb M}^2(\rho)$ are the following:
\begin{enumerate}
\item Case  $\mathbb{R}^2$. There are not critical curves.
\item Case $\mathbb{S}^2(\rho)$. Then $\mu\in\left(0,\sqrt{\rho}/2\right]$. Moreover, if $\mu<\sqrt{\rho}/2$, then there are two solutions, which correspond with two parallels and if $\mu=\sqrt{\rho}/2$, then the only solution is a  circle with curvature $\kappa_0=\sqrt{\rho}$.
\item Case  $\mathbb{H}^2(\rho)$. Then there always are two   critical curves with constant curvature, namely, a  circle and an hypercycle.
\end{enumerate}
\end{corollary}

\begin{proof}
From equation \eqref{EL}, if $\kappa=\kappa_0$ is constant then  
\begin{equation}
\kappa=\kappa_0=\frac{\rho\pm\sqrt{\rho^2-4\mu^2\rho}}{2\mu}. \label{constantcurvature}
\end{equation}
In particular, there are restrictions on the values of $\mu$ in relation with the constant $\rho$, namely, $\rho^2\geq 4\mu^2\rho$. It is clear that if $\rho=0$, then \eqref{constantcurvature} implies that $\kappa_0=0$, which is not possible.  

If $\rho>0$, then $\mu\leq\sqrt{\rho}/2$. When $\mu<\sqrt{\rho}/2$, we have two solutions, which correspond with two parallels of $\mathbb{S}^2(\rho)$. On the other hand, in the case that $\mu=\sqrt{\rho}/2$ there is only one solution that corresponds with a  circle of  curvature $\kappa_0=\sqrt{\rho}$.

Finally, if $\rho<0$, then the condition $\rho^2> 4\mu^2\rho$ is always satisfied, obtaining two critical curves with constant curvature. From \eqref{constantcurvature} we see that the curvature corresponding with the plus sign satisfies that $\kappa_0^2<-\rho$ and  the critical curve is an hypercycle. Similarly, for the minus sign we get $\kappa_0^2>-\rho$ and the curve  is a  circle.  
\end{proof}

We finish this section obtaining   a first integral of the Euler-Lagrange equation. Observe that from   \eqref{EL}, and after a change of orientation on $\gamma$, we can assume  $\mu>0$.  Set $b=\mu/\kappa$. Then equation \eqref{EL} writes as
\begin{equation}\label{E2}
(b^2)''+b'(b^2)'-2\left(\rho\left(1-b\right)-\frac{\mu^2}{b}\right)=0.
\end{equation}
 Define the function
$$ f(s)={\rm e}^{2b}\left(b^2\left(b'\right)^2+\rho\left(1-b\right)^2+\mu^2\right).$$
 The derivative of $f$ is  
\begin{eqnarray*}
f'(s)&=&{\rm e}^{2b}\left( 2\mu^2b'-\rho(1-b)\left(b^2\right)'+2b^2\left(b'\right)^3+\frac{1}{2}\left(b^2\right)'\left(b^2\right)''\right)\\&=& 2\mu b'{\rm e}^{2b}\left(\mu-b\kappa\right) =0,
\end{eqnarray*}
where \eqref{EL} has been used in the second identity. It follows that  there exists a constant $d\in\mathbb{R}$ such that $f(s)=d$.
Consequently, by the definition of $f$ and $b$, the derivative $\kappa_s$ of $\kappa$ satisfies
\begin{equation}
\kappa_s^2=\frac{\kappa^4}{\mu^4}\left(d\kappa^2{\rm e}^{-2\mu/\kappa}-\mu^2\kappa^2-\rho\left(\kappa-\mu\right)^2\right).\label{fin}
\end{equation}
 
We study the Euler-Lagrange equation \eqref{EL}. Let us introduce the following notation:
\begin{equation}\label{xx}
 x={\rm e}^{\mu/\kappa},\quad y=x_s.
 \end{equation}
Then \eqref{EL} writes as 
$$\frac{d^2}{ds^2}\left(\left(1-\log x\right)x\right)+\left(\rho\left(1-\log x\right)-\frac{\mu^2}{\log x}\right)x=0.$$
After some computations, this equation   is equivalent to the following system of ordinary differential equations 
\begin{equation}\label{phase}
\left\{\begin{array}{ll}
x'&=y\\
y'&=\dfrac{-\mu ^2 x^2-\rho  x^2 \log ^2x+\rho  x^2 \log x-y^2 \log x}{x \log ^2x}.
\end{array}\right.
\end{equation}
Using the first integral of the Euler-Lagrange equation \eqref{fin}, we   obtain the following result.

\begin{proposition} Let   $\gamma\subset {\mathbb M}^2(\rho)$ be a critical curve of the energy  $\mathbf{\Theta}_\mu$. Then there exists a constant $d\in\mathbb{R}$ such that $\gamma$ in coordinates \eqref{xx} satisfies\begin{equation}
F(x,y):=y^2\log^2x+\mu^2x^2+\rho\left(1-\log x\right)^2x^2=d,\quad x\in\mathbb{R}^+\setminus\{1\}. \label{F(x,y)}
\end{equation}
 In particular,  the constant $d$   is positive if $\rho\geq 0$.
\end{proposition}

Therefore, the implicit equation $F(x,y)=d$, \eqref{F(x,y)}, describes the orbits of the ODE system \eqref{phase}.

\section{Variational characterization of profile curves}\label{sec3}

A surface $S$   in ${\mathbb M}^3(\rho)$ is called  \emph{rotationally symmetric}, or simply a \emph{rotational surface}, if $S$ is invariant under the action of a one-parameter group of rotations of ${\mathbb M}^3(\rho)$. This group leaves a  geodesic $\alpha$ pointwise fixed called  the \emph{rotation axis}. In  $\mathbb{R}^3$ and $\mathbb{S}^3(\rho)$, any one-parameter group of rotations    is isomorphic to $\mathbb{S}^1$ (and the orbits are circles), but in the hyperbolic space $\mathbb{H}^3(\rho)$ there are three types of rotations (\cite{DoCarmo-Dajczer2}). As we will see, our interest in the hyperbolic space focuses only for   those  groups of rotations that are isomorphic to $\mathbb{S}^1$, which are called     \emph{spherical rotations}.

Let $S\subset {\mathbb M}^3(\rho)$ be a surface invariant by the group of rotations   $\{\phi_t: t\in\mathbb{R}\}$ and let $\gamma$ be a profile curve of $S$. Since $\gamma$ is a planar curve, we can assume that $\gamma$ lies  fully in a totally geodesic surface of ${\mathbb M}^3(\rho)$, which we identify with a $2$-space form ${\mathbb M}^2(\rho)$. We prove that   $\gamma$ has a variational characterization if $S$ is a surface of constant astigmatism.
 
\begin{theorem}\label{vc} Let $S$ be a rotational surface in ${\mathbb M}^3(\rho)$ satisfying  the constant astigmatism equation \eqref{relation}. If $\gamma$ is a profile curve of $S$, then   $\gamma$ satisfies the Euler-Lagrange equation \eqref{EL} of the curvature energy
 \begin{equation} \mathbf{\Theta}_\mu(\gamma)=\int_\gamma \kappa\, {\rm e}^{\mu/\kappa}\,ds, \label{energy0} \end{equation}
where $\mu=1/c$ and $\kappa$ is the curvature of $\gamma$.
\end{theorem}

\begin{proof}
 Denote by $\eta$ the unit normal vector field of $S$ and let $\xi$ be  the Killing vector field which is the infinitesimal generator of the rotations that leave $S$ invariant. Because the result is local, let us take   Fermi geodesic coordinates  
$X:U\subset\mathbb{R}^2\rightarrow S$ on $S$, 
\begin{equation}
X(s,t)=\phi_t\left(\gamma(s)\right),\quad \xi=\frac{\partial}{\partial t},\label{par}
\end{equation} 
where $s$ measures the arc-length
along geodesics orthogonal to $\xi$ and 
$\gamma(s)=X(s,0)$.   By \eqref{par},  the curve  $\gamma$ and
all its copies   
$\gamma_t(s)=\phi_t(\gamma(s)) $  are arc-length
parametrized planar geodesics of $S_\gamma$ which are orthogonal to
$\xi$.  Moreover, $\gamma$ and all $\gamma_t$ are not geodesics   in
${\mathbb M}^3(\rho)$ by \eqref{relation} and thus   $\gamma_t(s)$ are Frenet curves. Denote the  Frenet frame of $\gamma_t$ as $\{T(s,t),N(s,t),B(s,t)\}$ and by $\kappa(s,t)$ its curvature in ${\mathbb M}^3(\rho)$.

On the other hand,    the Gauss-Codazzi equations  of $S$ in ${\mathbb M}^3(\rho)$ have a simple expression, namely, 
\begin{equation}
 \frac{\partial}{\partial s}\left(\frac{1}{\kappa}\left(G_{ss}+G(\kappa^2+\rho)\right)\right) -\kappa_s G=0, \label{Gauss-Codazzi}
\end{equation}
where $G$ is the length of the Killing vector field $\xi$, that is, $G(s)^2=\langle X_t,X_t\rangle$. Here we refer the reader to \cite[Sec. 3]{AGP}  for details. Furthermore, not only $G(s)$, but also all the involved functions depend only on $s$. The principal curvatures of $S$ are $\kappa_1(s)=h_{22}(s)$, the second coefficient of the second fundamental form, and    $\kappa_2(s)=-\kappa(s)$. Since
\begin{equation}\label{h22}
h_{22}=\frac{1}{\kappa}\left(\frac{G_{ss}}{G}+\rho\right),
\end{equation}
  the relation \eqref{relation} becomes
\begin{equation}
\left(c\kappa-1\right)G_{ss}=G\left(\kappa^2-\rho\left(c\kappa-1\right)\right) .\label{three}
\end{equation}
First, we study the particular case that $\kappa$ is a constant function  $\kappa(s)=\kappa_0$. By \eqref{relation},   $\kappa_0\not=1/c$. Thus, we can combine (\ref{Gauss-Codazzi}) and
(\ref{three}) to deduce that $G(s)$ must be a positive constant, hence  $S$ is a flat isoparametric surface. From \eqref{three}, we deduce that 
\begin{equation}
\kappa_0^2=\rho\left(c\kappa_0-1\right). \nonumber
\end{equation}
Therefore, by equation \eqref{constantcurvature}, the curve $\gamma$ is a critical curve with constant curvature of $\mathbf{\Theta}_\mu$  for $\mu=1/c$, and the theorem is proved in this case.

Suppose now that the curvature $\kappa$ of $\gamma$ is not constant. By the
Inverse Function Theorem, we can suppose that $s$ is locally a function of
$\kappa$. Set $G(\kappa)=\dot{P}(\kappa)$, where the
upper dot denotes derivative with respect to $\kappa$. We   integrate once equation \eqref{Gauss-Codazzi}, obtaining that there is a constant $\lambda\in \mathbb{R}$ such that 
\begin{equation}\label{gc}
\dot{P}_{ss}+\dot{P}\left(\kappa^2+\rho\right)-\kappa \left(P+\lambda\right)=0.
\end{equation}
On the other hand, equation \eqref{three} can be  now expressed as
$$\dot{P}_{ss}-\dot{P}\left(\frac{\kappa^2}{c\kappa-1}-\rho\right)=0.$$
By substituting in this equation the value of   $\dot{P}_{ss}$ obtained in  \eqref{gc}, we find
$$\frac{c\kappa^3}{c\kappa-1}\dot{P}-\kappa P-\lambda\kappa=0.$$
A direct integration gives 
\begin{equation}\label{pkk}
P(\kappa)=\kappa\, {\rm e}^{\mu/\kappa}-\lambda,
\end{equation}
where $\mu=1/c$. For this function  $P(\kappa)$, equation \eqref{gc} is just the Euler-Lagrange equation \eqref{EL}, concluding that   $\gamma$ is  a critical curve of $\mathbf{\Theta}_\mu$. This proves  the result. 
\end{proof}

\section{Geometric description of the rotational surfaces}\label{sec4}

In this section we will prove the converse of Theorem \ref{vc} by evolving extremal curves under their associated binormal flow. This will give us in Theorem \ref{converse} a way of constructing all   rotational surfaces of constant astigmatism in space forms. 

First, we   see that the critical curves of $\mathbf{\Theta}_\mu$ have a distinguished vector field along them. Let $\gamma=\gamma(s)$ be a curve in ${\mathbb M}^2(\rho)$ that is a critical curve of $\mathbf{\Theta}_\mu$. A vector field $W$ along $\gamma$, which infinitesimally
preserves unit speed parametrization, is called a \emph{Killing
vector field along $\gamma$} in the sense of Langer-Singer (\cite{Langer-Singer}), 
if $\gamma$ evolves in the direction of $W$ without changing
shape, only its position. In other words,  the following equations
must hold
\begin{equation}
W(v)(s,0)=W(\kappa)(s,0)=0, \nonumber
\end{equation}
for any variational vector field of $\gamma$ having $W$ as
variation vector field. Here $ v=\lvert \gamma'\rvert$ is
the speed of $\gamma$. 

It turns out that the critical curves  of
$\mathbf{\Theta}_\mu$ have a natural associated Killing vector field defined along
$\gamma$. Define the vector field $\mathcal{I}$ along
$\gamma$ by
\begin{equation}
\mathcal{I}=T\times\mathcal{K}=\left(1-\frac{\mu}{\kappa}\right){\rm e}^{\mu/\kappa}
B,\label{I}
\end{equation}
where $\mathcal{K}$ is
defined in \eqref{kfield}. Here $B=T\times N$ is the binormal of $\gamma$ viewing ${\mathbb M}^2(\rho)$ inside ${\mathbb M}^3(\rho)$ and $\gamma$ with zero torsion, being $B$ a constant vector field. The following result is proved in \cite[Prop. 3]{AGP}.

\begin{proposition}\label{carcenter}
If $\gamma\subset {\mathbb M}^2(\rho)$ is a critical curve of the energy  
$\mathbf{\Theta}_\mu$, then 
$$
\mathcal{I}=\left(1-\frac{\mu}{\kappa}\right){\rm e}^{\mu/\kappa}B$$
 is  a Killing vector field along $\gamma$.
\end{proposition}

From this result  and using an argument similar as in \cite{Langer-Singer},  we   extend $\mathcal{I}$  to a Killing
vector field in  the ambient space ${\mathbb M}^3(\rho)$, and that we denote by $\mathcal{I}$ 
again. Since ${\mathbb M}^3(\rho)$ is complete, we   consider the
one-parameter group of isometries $\{\phi_t: t\in \mathbb{R}\}$ determined by the flow of
$\mathcal{I}$, and define the
surface $S_\gamma=\{\phi_t(\gamma(s)):s\in I, t\in\mathbb{R}\}$ obtained as the
evolution of $\gamma$ under the $\mathcal{I}$-flow. Observe that
$S_\gamma$ is an $\mathcal{I}$-invariant surface, and $S_\gamma$ is foliated by congruent
copies of $\gamma$, namely, $\gamma_t(s)=\phi_t(\gamma(s))$.

If $X(s,t)=\phi_t(\gamma(s))$ is a  parametrization of $S_\gamma$, and because $\phi_t$ are isometries of ${\mathbb M}^3(\rho)$, we have
$$X_t(s,t)= \left(1-\frac{\mu}{\kappa}\right){ \rm e}^{\mu/\kappa} B(s,t),$$
where $\kappa$ is the curvature of $\gamma$ and $B(s,t)$ is the
unit   binormal vector of $\gamma_t(s)$. Thus  $S_\gamma$   is
a \emph{binormal evolution surface} with velocity
$V(s)=\langle X_t,X_t\rangle^{1/2}
=\langle\mathcal{I},\mathcal{I}\rangle^{1/2}$.

Following \cite[Prop. 3]{AGP},  and since all the  filaments of $S_\gamma$  satisfy $\tau=0$,  the fibers of $S_\gamma$ have constant curvature and zero torsion  in ${\mathbb M}^3(\rho)$. In the particular case that  the curvature $\kappa(s,t)$ of all filaments is   constant, then $S_\gamma$ is a flat isoparametric surface, so  $S_\gamma$ is a right circular cylinder because $S_\gamma$ is not totally umbilical flat (\cite{sp}). For the case where the filaments have non-constant curvature,   and since   the curvature of $\gamma$ is not constant, then the surface $S_\gamma$ is a rotational surface, as it can be proved adapting the computations of \cite[Prop. 5.3]{AGP2}. We summarize this result in the following proposition.

\begin{proposition}\label{rot} 
Let   $\gamma\subset {\mathbb M}^2(\rho)$ be a critical curve of the energy  $\mathbf{\Theta}_\mu$.  Let $S_\gamma$ be a binormal evolution surface in ${\mathbb M}^3(\rho)$ such that all   filaments have zero torsion. If  all filaments   have non-constant curvature, then $S_\gamma$ is a rotational surface. 
\end{proposition}

We are in conditions to prove the converse of Theorem \ref{vc}.

\begin{theorem}\label{converse} Let $\gamma\subset {\mathbb M}^2(\rho)$ be a critical curve of the energy $\mathbf{\Theta}_\mu$  with non-constant curvature and let $S_\gamma$ denote the $\mathcal{I}$-invariant surface in ${\mathbb M}^3(\rho)$ obtained by evolving $\gamma$ under the flow of the Killing field $\mathcal{I}$ which extends \eqref{I} to ${\mathbb M}^3(\rho)$. Then  $S_\gamma$ is a rotational surface of ${\mathbb M}^3(\rho)$ satisfying the constant astigmatism equation \eqref{relation}   for the value $c=1/\mu$.
\end{theorem}

\begin{proof}
We   locally define the $\mathcal{I}$-invariant surface $S_\gamma=\{\phi_t(\gamma(s)\}$ where $\{\phi_t: t\in\mathbb{R}\}$ is the one-parameter group of isometries determined by $\mathcal{I}$. Furthermore, the square of the length of the Killing vector field $\mathcal{I}$ is  
\begin{equation}
V(s)^2=\langle\mathcal{I},\mathcal{I}\rangle=\frac{\left(\kappa-\mu\right)^2}{\kappa^2}{\rm e}^{2\mu/\kappa}. \label{Gs}
\end{equation}
Since  the evolution is done by isometries, the curve $\gamma$ and all its congruent copies are planar critical curves of $\mathbf{\Theta}_\mu$ and Proposition \ref{rot} implies  that $S_\gamma$ is a rotational surface. Finally, any $\gamma_t$ satisfies the Euler-Lagrange equation \eqref{EL} which is, using \eqref{Gs}, equivalent to
$$\left(\kappa-\mu\right)V_{ss}=\left(\mu\kappa^2-\rho\left(\kappa-\mu\right)\right)V.$$
Because the principal curvatures of $S_\gamma$ are $\kappa_1=h_{22}(s)$ and $\kappa_2=-\kappa(s)$, we deduce that $S_\gamma$ is a surface  of constant astigmatism in ${\mathbb M}^3(\rho)$ with     $c=1/\mu$ in  \eqref{relation}. 
\end{proof}

   Theorem \ref{converse} provides a method  of constructing rotational surfaces  of constant astigmatism in   ${\mathbb M}^3(\rho)$. Together  with Theorem \ref{vc}, we characterize   all these surfaces as binormal evolution surfaces generated by  extremal planar curves of $\mathbf{\Theta}_\mu$. In conclusion, we have that a  rotational surface of constant astigmatism  in a    space form  must be either totally umbilical, a right circular cylinder
constructed over an extremal curve with constant curvature (Corollary \ref{ccc}) or a binormal evolution
surface generated by  extremal planar curves of $\mathbf{\Theta}_\mu$.

\begin{remark} When ${\mathbb M}^3(\rho)=\mathbb{S}^3(\rho)$, the constant astigmatism flat isoparametric   rotational surfaces (right circular cylinders) are Hopf tori given by the product 
$$\mathbb{S}^1\left(\frac{\sqrt{2}\mu}{m}\right)\times\mathbb{S}^1\left(\frac{\sqrt{m^2-2\mu^2\rho}}{\sqrt{\rho}\, m}\right),$$ where $m^2=\lvert \rho^2\pm\sqrt{\rho^2-4\mu^2\rho}\rvert$. The particular case  $\mu=\sqrt{\rho}/2$ is the Clifford torus, which is also a minimal surface.
\end{remark}


\section{Analysis of the extremal curves}\label{sec5}

Previously to the classification of the profile curves of the rotational surfaces of constant astigmatism,  in this section we study  some  geometric properties of the critical curves of $\mathbf{\Theta}_\mu$. The key point is to obtain  an useful expression of the parametrization of these curves in terms of its curvature. Consider ${\mathbb M}^2(\rho)$ viewed as a subset of the affine space $\mathbb{R}^3$ with canonical coordinates $(x_1,x_2,x_3)$, that is, $\mathbb{R}^2$ is $\mathbb{R}^2\times\{0\}$, the sphere ${\mathbb S}^2(\rho)$ is $x_1^2+x_2^2+x_3^2=1/\rho$ and $\mathbb{H}^2(\rho)$ is $x_1^2+x_2^2-x_3^2=1/\rho$, $x_3>0$. Here $\mathbb{R}^2$ and $\mathbb{S}^2(\rho)$ are  endowed with the Euclidean metric of ${\mathbb R}^3$ and $\mathbb{H}^2(\rho)$ is equipped with the induced metric of the   Lorentzian metric $dx_1^2+dx_2^2-dx_3^3$ of $\mathbb{R}^3$. 

In the case $\rho<0$, that is, in the hyperbolic space, we now prove that the case $d>0$ in \eqref{F(x,y)} corresponds with rotational surfaces of spherical type.    Exactly, the three possible signs of the constant $d$ in \eqref{F(x,y)} indicates the three types of rotational surfaces in the hyperbolic space. We see this as follows.  With the notation of Section \ref{sec4} and by  \cite[equation (47)]{AGP},   the curvature $\kappa_\delta$ of the orbit of $S_\gamma$ satisfies
$$\kappa_\delta^2=\frac{\dot{P}_{s}^2}{\dot{P}^2}+h_{22}^2.$$
By combining \eqref{fin}, \eqref{h22} and \eqref{pkk}, we find
$$\kappa_\delta^2=\frac{d\kappa^2 {\rm e}^{-2\mu/\kappa}}{(\mu-\kappa)^2}-\rho.$$
Since $\rho<0$, we deduce    $\kappa_\delta>-\rho$, which implies that the orbits are circles, if and only if $d>0$. Definitively, the choice $d>0$ in \eqref{F(x,y)} corresponds with rotational surfaces   of spherical type.

In this section we suppose that the constant $d$ in \eqref{F(x,y)} is positive: recall that $d$ is always positive when $\rho\geq 0$.  Define the function $\phi:\mathbb{R}^2\rightarrow {\mathbb M}^2(\rho)\subset\mathbb{R}^3$ by
\begin{eqnarray}
\phi(u,v)=\left\{ \begin{array}{lcc} \dfrac{1}{\sqrt{d}}\left(u,dv,0\right) & \text{if} & \rho=0 \\ 
\dfrac{1}{\sqrt{d}}\left(u,\sqrt{d-\rho u^2} \sin (\sqrt{\rho d}\,v), \sqrt{d-\rho u^2}\cos (\sqrt{\rho d}\,v)\right)\ & \text{if} & \rho>0 \\ \dfrac{1}{\sqrt{d}}\left(u,\sqrt{d-\rho u^2}\sinh (\sqrt{-\rho d}\,v),\sqrt{d-\rho u^2}\cosh (\sqrt{-\rho d}\,v)\right)& \text{if} & \rho<0. \end{array}\right.\label{tildephi}
\end{eqnarray}
As it turns out, up to rigid motions, it is always possible to find a coordinate system in ${\mathbb M}^2(\rho)\subset\mathbb{R}^3$ such that critical curves of any curvature energy verify that their first component is a multiple of $\dot{P}(\kappa)$. It follows  that any critical curve $\gamma$ of $\mathbf{\Theta}_\mu$ can be parametrized, up to rigid motions, as 
\begin{equation}\label{gg}
\gamma(s)=\phi\left((1-\log x)x, \psi(s)\right), \quad x={\rm e}^{\mu/\kappa},
\end{equation}
where the function $\psi(s)$ comes from the parametrization by arc-length. Without loss of generality, we can assume that the arc-length parameter  $s$  is chosen so that $x(0)=x_{0}$, hence
\begin{equation}
\psi(s)=-\mu\int_0^s \frac{x(t)}{d-\rho\left(1-\log x(t)\right)^2x^2(t)}\,dt. \label{psi(s)}
\end{equation}
Notice that the denominator can not vanish  because if  $d=\rho(1-\log{x})^2x^2$ at some point,   from   \eqref{F(x,y)}  we find $y^2\log^2{x}+\mu^2x^2=0$, which is not possible.  The parametrization \eqref{gg} allows to  obtain some geometric properties about critical curves of $\mathbf{\Theta}_\mu$ concerning symmetries and cuts with respect to some coordinate axis. First we prove that the critical curves are symmetric about a fixed axis.  

\begin{proposition}\label{sym} Up to a rigid motion, any critical curve of $\mathbf{\Theta}_\mu$ in ${\mathbb M}^2(\rho)$ is symmetric with respect to the geodesic $\alpha= \Pi_{13}\cap {\mathbb M}^2(\rho)$, where $\Pi_{13}$ is the plane of equation $x_2=0$. 
\end{proposition}

\begin{proof} If $\gamma(s)$ is a  critical curve of $\mathbf{\Theta}_\mu$, a change of variables in  \eqref{psi(s)} gives  $\psi(-s)=-\psi(s)$. Therefore, using the parametrization \eqref{gg}, we conclude that   the second coordinate of $\gamma(-s)$ change of sign with  the corresponding one of $\gamma(s)$ and the first and third coordinates coincide.  This finishes the proof. 
\end{proof}

On the other hand, we   obtain the points where extremal curves of $\mathbf{\Theta}_\mu$  cut or tend to cut the other coordinate axis.

\begin{proposition}\label{cuts} Consider the geodesic of ${\mathbb M}^2(\rho)$ given by $\beta=\Pi_{23}\cap {\mathbb M}^2(\rho)$ where $\Pi_{23}$ is the plane of equation $x_1=0$. Then, up to a rigid motion, any critical curve $\gamma$ of $\mathbf{\Theta}_\mu$ may meet $\beta$ at two   types of points, namely, 
\begin{enumerate}
\item Regular points. The curve $\gamma$ cuts the geodesic $\beta$ at points satisfying $x={\rm e}$.
\item Singular points. If $x$ tends to zero, then $\gamma$ tends to meet $\beta$.
\end{enumerate}
\end{proposition}

\begin{proof} 
By the   parametrization \eqref{gg},  $\gamma$ meets the geodesic $\beta$ whenever its first component vanishes, that is,  when $\left(1-\log x\right)x=0$. If $x={\rm e}$, then the  intersection occurs at a regular point of $\gamma$. If    $x=0$   then the point is outside of  the parametrization of $\gamma$, although it may be added in order to obtain complete curves. In this case, if   $x$ tends to $0$, then $\gamma$  tends to $\beta$.
\end{proof}

We now prove a result on the uniqueness of critical curves of $\mathbf{\Theta}_\mu$ in ${\mathbb M}^2(\rho)$. 

\begin{proposition}\label{unique} For each fixed $d>0$, critical curves of $\mathbf{\Theta}_\mu$  in ${\mathbb M}^2(\rho)$ are unique up to change of orientation. In the Euclidean plane $\mathbb{R}^2$, critical curves of $\mathbf{\Theta}_\mu$  are also unique after dilations.
\end{proposition}

\begin{proof}
Let $\gamma=\gamma(s)$ be a critical curve of  $\mathbf{\Theta}_\mu$ and we write \eqref{F(x,y)} as   $F(x,y;\mu)=d$ to indicate the dependence of $F$ on   the energy index $\mu$. We reverse the orientation on $\gamma$, say, $\widetilde{\gamma}(s)=\gamma(-s)$. If    $\widetilde{\kappa}$ is  the curvature of $\widetilde{\gamma}$, then   $\widetilde{\kappa}(-s)=-\kappa(s),$  hence  
\begin{eqnarray*}
x(s;\mu)&=&\widetilde{x}(-s;-\mu)\\
y(s;\mu)&=&-\widetilde{y}(-s;-\mu).
\end{eqnarray*}
Using the above equations,  we see that  
$$F(x,y;\mu)=F(\widetilde{x},\widetilde{y};-\mu)=d,$$
proving that the  critical curves for $\mu$ and for $-\mu$ are the same.

In the Euclidean plane,  if we apply a dilation of ratio $\lambda>0$ to $\gamma$, say $\widetilde{\gamma}(s)=\lambda\gamma(s)$, then its   curvature $\widetilde{\kappa}$ satisfies $\widetilde{\kappa}(s/\lambda)=\lambda\kappa(s)$. In this case, we obtain  
\begin{eqnarray*}
x(s;\mu)&=&\widetilde{x}(s/\lambda;\lambda\mu)  \\
y(s;\mu)&=&\frac{1}{\lambda}\widetilde{y}(s/\lambda;\lambda\mu).
\end{eqnarray*} 
Since $\rho00$, it follows that
$$F(x,y;\mu)=\frac{1}{\lambda^2}F(\widetilde{x},\widetilde{y};\lambda\mu)=d$$
Thus the  critical curves of $\mathbf{\Theta}_\mu$ are the same of $\mathbf{\Theta}_{\lambda\mu}$ for the constant of integration $\widetilde{d}=\lambda^2d$. 
\end{proof}

For dilations, the result does not hold if $\rho\not=0$ because  after applying the corresponding relations between $x$, $y$, $\widetilde{x}$ and $\widetilde{y}$ as before, we obtain that the sectional curvature $\rho$ of ${\mathbb M}^2(\rho)$  should also be deformed to $\lambda\rho$, which is not possible since it must be fixed from the beginning.

In the expression \eqref{psi(s)} of $\psi(s)$,   we use equation $y=x_s$ and \eqref{F(x,y)} to make a change of variable obtaining 
 \begin{equation} 
\psi(x)=\mu\int_{x}^{x_{0}} \frac{r\log r}{\left(d-\rho\left(1-\log r\right)^2r^2\right)\sqrt{d-\mu^2r^2-\rho\left(1-\log r\right)^2r^2}}\,dr.\label{psi(x)}
\end{equation}
The function $\psi(x)$ in   \eqref{psi(x)} (or  $\psi(s)$ in \eqref{psi(s)}) provides information of the critical curve $\gamma$ thanks to \eqref{gg}. For example, because  $d>\rho\left(1-\log x\right)^2x^2$, then   $\psi'(x)>0$ if  $x\in\left(0,1\right)$ and $\psi'(x)<0$ if   $x>1$. This implies that  $\psi(x)$ monotonically increases when $x\in (0,1)$ and   $\psi(x)$  decreases if $x>1$. For   fixed $\rho$ and $\mu$, set $\psi(x)=\psi(x;d)$ to indicate the dependence on $d$. Consider the value $\lim_{x\rightarrow 0}\psi(x;d)$ as a function  depending on the parameter $d$ and let $d_*$ be the only value of $d$ such that $\lim_{x\rightarrow 0}\psi(x;d)=0$.

In order to  study the     ODE system \eqref{phase}, we analyze the corresponding phase plane. Let $Q(x,y)$ be  the    tangent vector field  
$$Q(x,y)=\left(y,\frac{-\mu ^2 x^2-\rho  x^2 \log x(\log x-1)-y^2 \log x}{x \log ^2x}\right),$$
which is defined in $(\mathbb{R}^+\setminus\{1\})\times\mathbb{R}$. The singular points of $Q$ are the points of the form  $(x,0)$, where $x$ satisfies  
\begin{equation}
\mu^2=\rho\log x\left(1-\log x\right). \label{cp}
\end{equation}
In particular,  there are not singular points if $\rho=0$. In figure \ref{orbitasr2}, left, we plot the phase plane for $\rho=0$.  Assume that $\rho\neq 0$. From \eqref{cp} and if $\rho^2-4\mu^2\rho\geq 0$, we find
\begin{equation}\label{slog}
\log x=\frac{\rho\pm\sqrt{\rho^2-4\mu^2\rho}}{2\rho}.
\end{equation}
Thus  there are two singular points $P_{+}=(x_{+},0)$ and $P_{-}=(x_{-},0)$, which may coincide, and  correspond  with the choices $+$ and $-$  in \eqref{slog}, respectively.  After some computations, at the  singular points  we have
\begin{equation}\label{hes}
\left(\begin{array}{ll}\dfrac{\partial Q_1}{\partial x}&\dfrac{\partial Q_1}{\partial y}\\
\dfrac{\partial Q_2}{\partial x}&\dfrac{\partial Q_2}{\partial y}\end{array}\right)(x_{\pm},0)=
\left(\begin{array}{cc}0&1\\ \mp\dfrac{\sqrt{\rho^2-4\mu^2\rho}}{\log^2x}&0\end{array}\right).
\end{equation}

\begin{figure}[hbtp]
	\begin{center}\includegraphics[width=.4\textwidth]{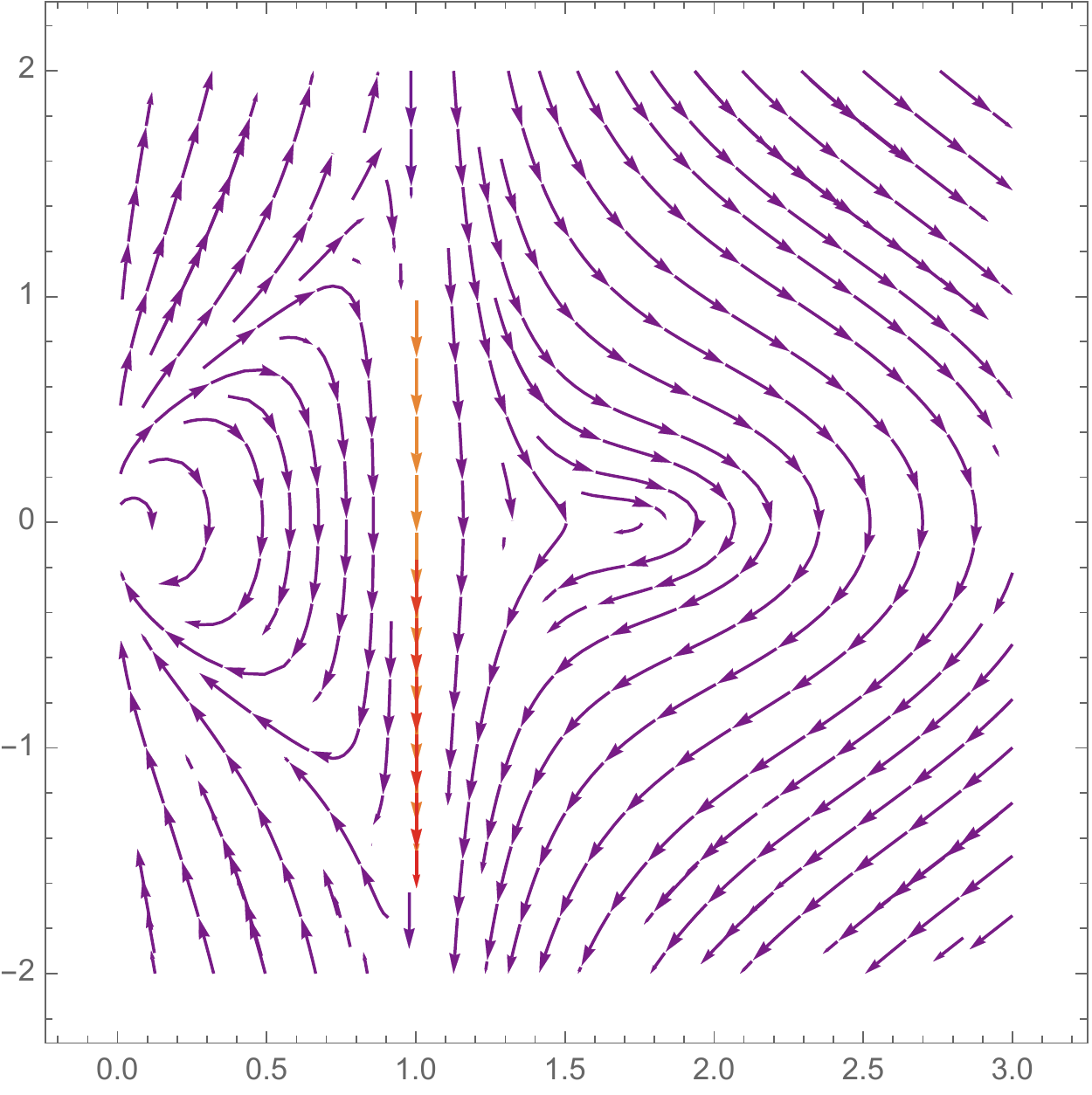}\quad \includegraphics[width=.4\textwidth]{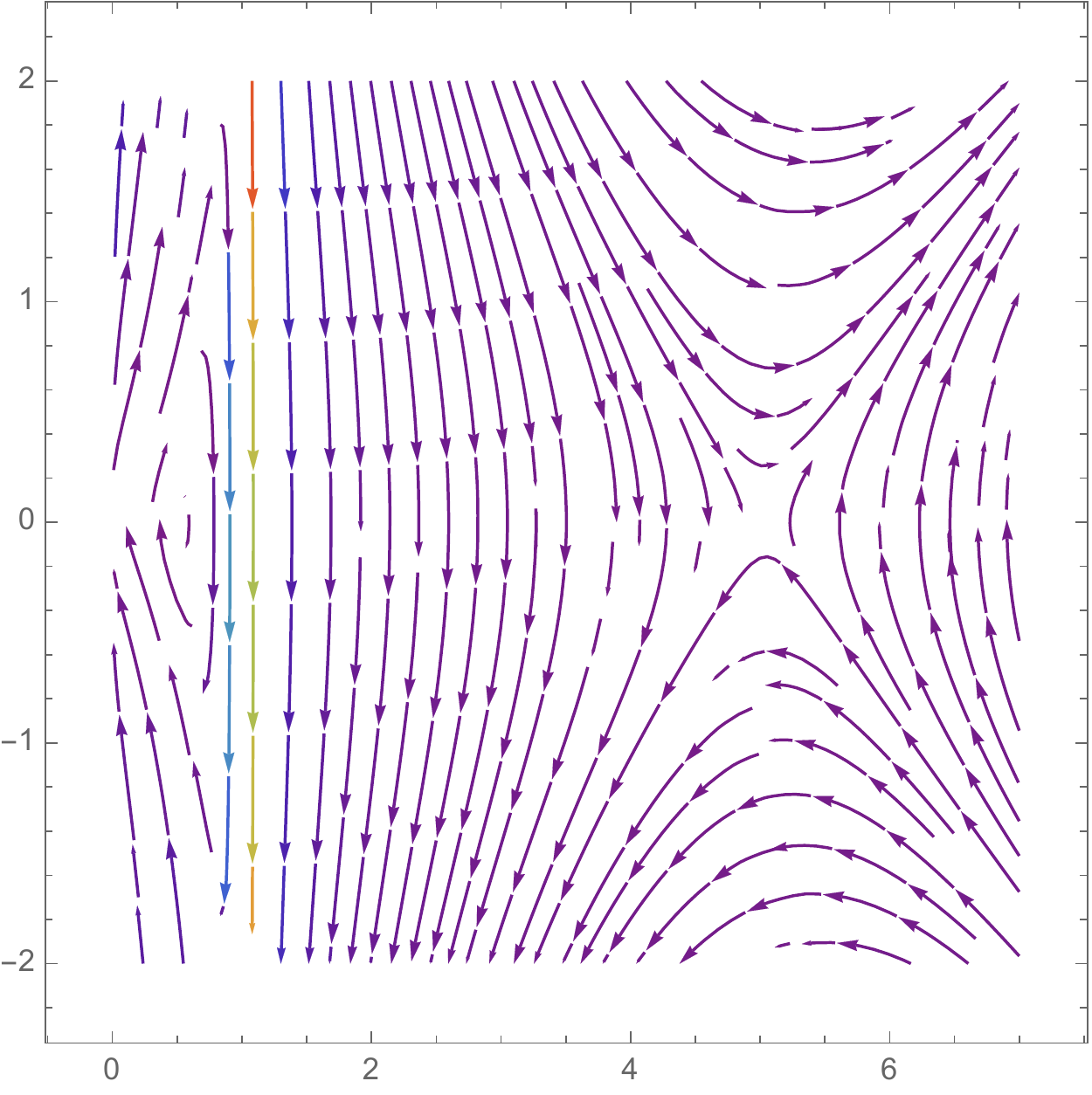}
	\end{center}
		\caption{(Left) Phase plane in the Euclidean plane $\mathbb{R}^2$. (Right) Phase plane in the hyperbolic plane $\mathbb{H}^2(\rho)$: here $\rho=-1$, $\mu=-1$ and the singular points are $P_{+}=(0.54,0)$ (centre) and $P_{-}=(5.04,0)$ (saddle point)}
		\label{orbitasr2}
\end{figure}

The analysis of the type of the singular points depends on the eigenvalues of \eqref{hes}, which  are
\begin{equation}\label{eigen}
\lambda_1=-\sqrt{2 \mu ^2-\rho  \log (x)},\ \lambda_2=\sqrt{2 \mu ^2-\rho  \log (x)}.
\end{equation}
We study the case $\rho>0$. From \eqref{slog}, there are not singular points when $\rho<4\mu^2$. If   $\rho=4\mu^2$,   there is a unique singular point  $P_{+}=P_{-}=P=\left(\sqrt{{\rm e}},0\right)$, where  the matrix \eqref{hes} is not diagonalizable with zero eigenvalue,   so $P$ is a degenerate point. Finally, if $\rho>4\mu^2$, we have two distinct singular points, $P_{+}$ and $P_{-}$, with $x_{-}<x_{+}$. By \eqref{eigen}, the eigenvalues for $P_+$ are two distinct pure imaginary complex numbers and the eigenvalues for $P_{-}$ are two distinct real numbers with different sign. Thus $P_+$ is a center and $P_{-}$  is an unstable  saddle point. See the phase plane in figure \ref{sp}. 

\begin{figure}[hbtp]
	\begin{center}\includegraphics[width=.32\textwidth]{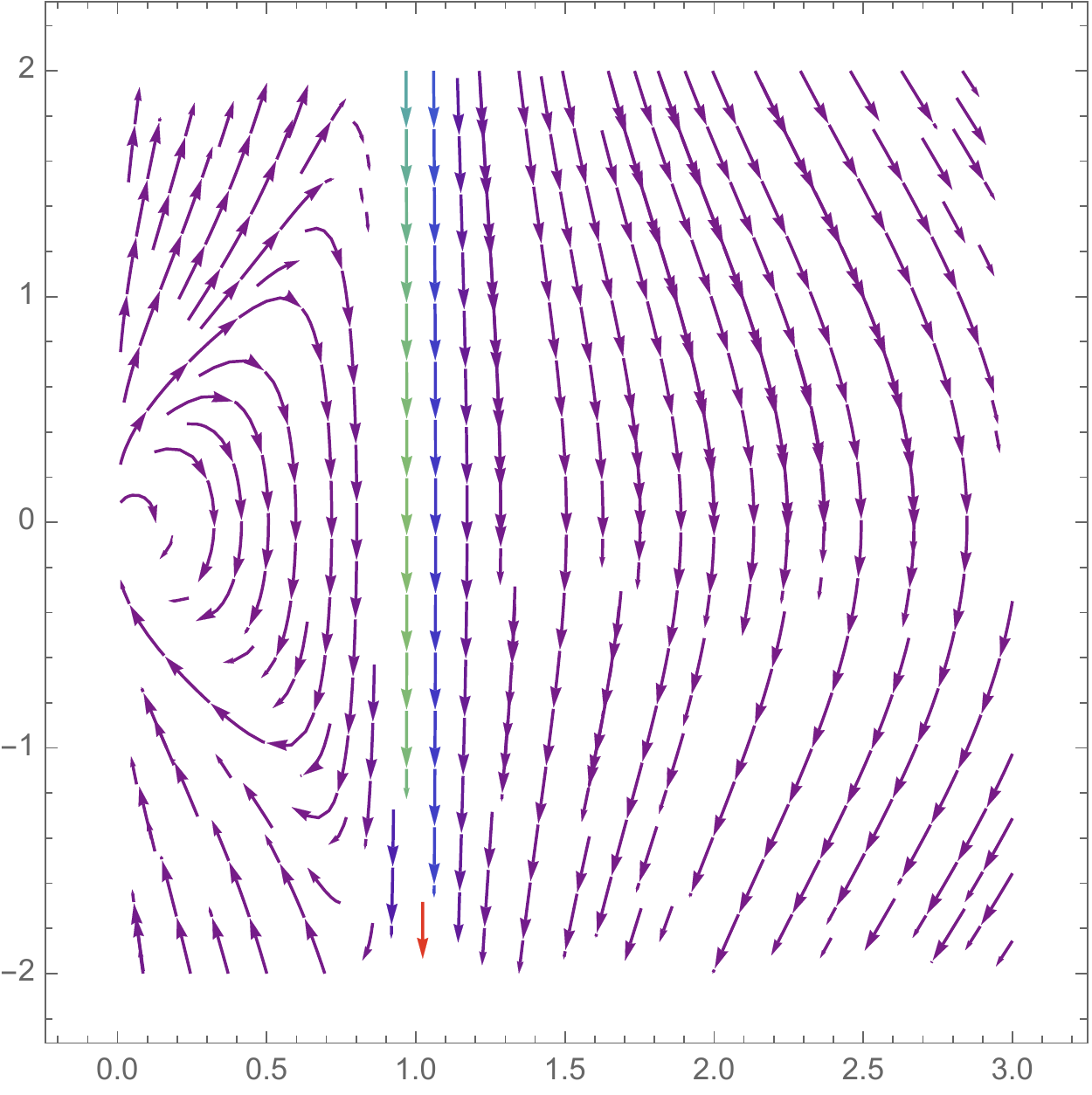}\includegraphics[width=.32\textwidth]{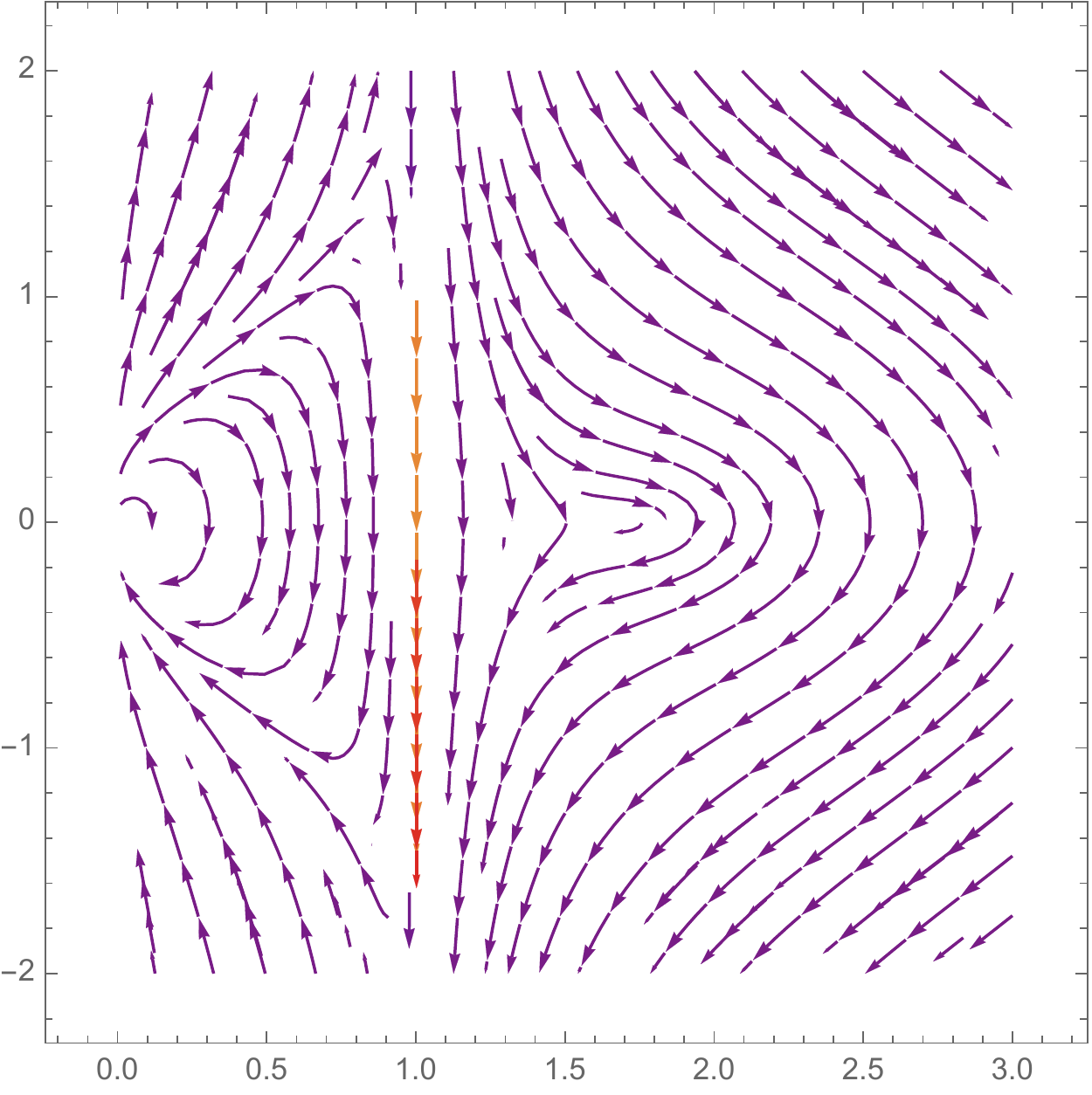}\includegraphics[width=.32\textwidth]{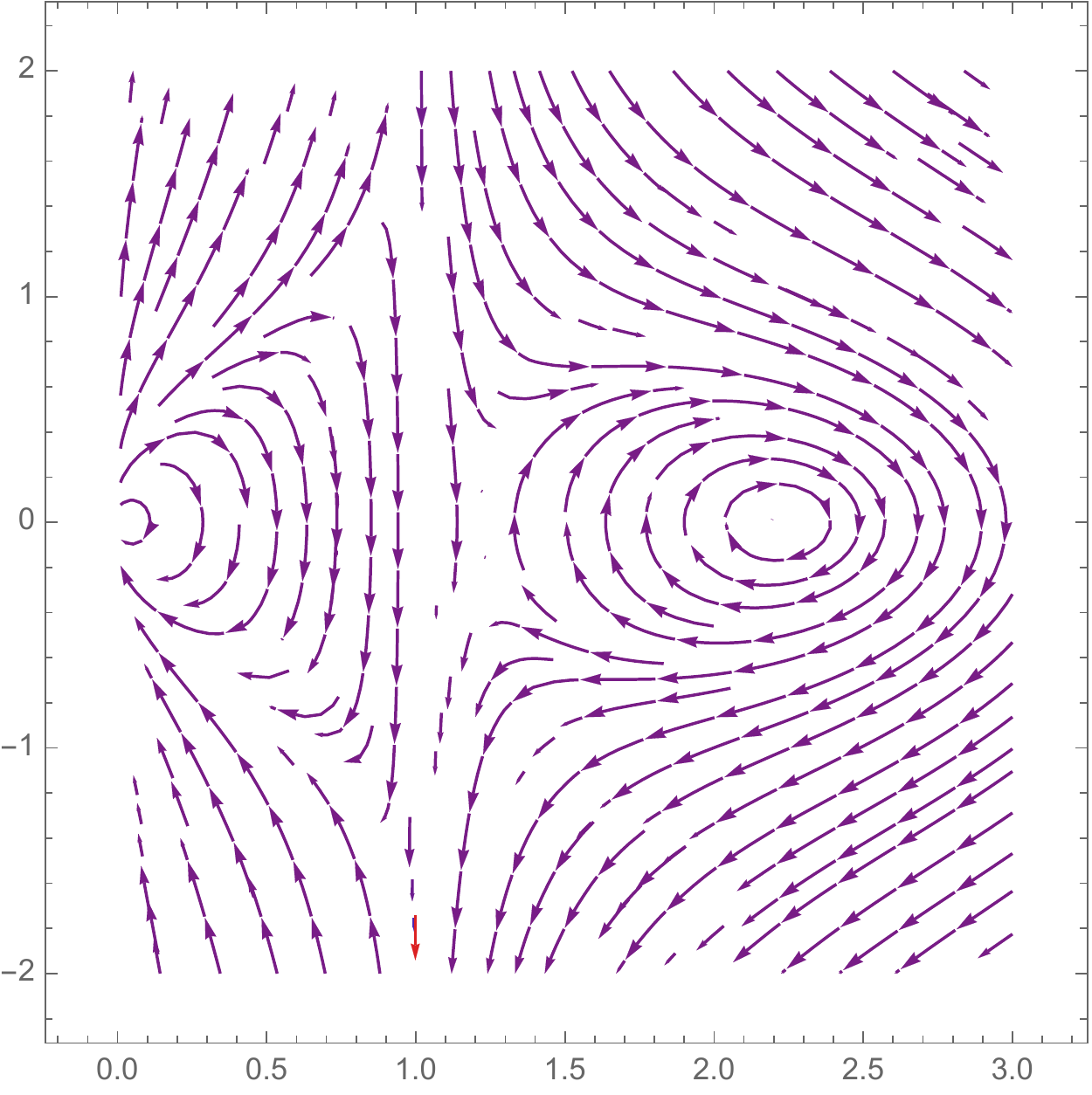}\end{center}
		\caption{Phase plane in the sphere $\mathbb{S}^2(\rho)$. (Left) Case $\rho<4\mu^2$: there are not singular points. (Middle) Case $\rho=4\mu^2$: the point $P=(\sqrt{{\rm e}},0)$ is a degenerate singular point. (Right) Case $\rho>4\mu^2$:  here $\rho=1$ and $\mu=0.4$ and the singular points are $P_{-}=(1.22,0)$, which is a saddle point, and $P_{+}=(2.22,0)$, which is a center}
		\label{sp}
\end{figure}

Let us see the case  $\rho<0$. By \eqref{slog} and since  $\rho^2-4\mu^2\rho$ is always positive, then there are two singular points,  $P_{+}=(x_+,0)$ and $P_{-}=(x_{-},0)$ where now   $x_{+}<1<x_{-}$. Similarly as in the   case $\rho>0$,  the eigenvalues corresponding for $P_{+}$ are two opposite pure imaginary complex numbers, so $P_{+}$ is a centre, whereas for $P_{-}$ the eigenvalues are   two  real numbers with opposite sign, hence    $P_{-}$ is an unstable  saddle point: see the phase plane in figure \ref{orbitasr2}, right.

We finish this section by analyzing  the  existence of closed critical curves. Notice that the problem of existence  of closed critical curves is  a difficult matter.   Here  we establish what is the  relation   between the value $d$ in \eqref{F(x,y)} and the energy index $\mu$ of the energy  $\mathbf{\Theta}_\mu$.

\begin{proposition}\label{closed} Let $d>0$. Then there are no closed critical curves of $\mathbf{\Theta}_\mu$ in $\mathbb{R}^2$ and $\mathbb{H}^2(\rho)$. In   $\mathbb{S}^2(\rho)$, if   $\gamma$  is a closed critical curve of $\mathbf{\Theta}_\mu$, then the number
\begin{equation}
\mu\sqrt{\rho d}\int_0^\varrho \frac{x(t)}{d-\rho(1-\log x(t))^2x^2(t)}\,dt\label{I(d)}
\end{equation}
is a rational multiple of $2\pi$ for some $d\in(\rho x_{+}^2\log x_{-},\rho x_{-}^2\log x_{+})$, where $\varrho$ is the period of the curvature $\kappa$.
\end{proposition}

\begin{proof} 
Since closed curves have periodic curvatures,  we study  when  equation \eqref{fin} admits periodic solutions, that is, closed orbits in the phase plane \eqref{phase}.  Closed orbits only appear in a local neighborhood of centre points, and consequently, we   restrict   to the case   $\rho\neq 0$ and $\rho^2-4\mu^2\rho>0$: this excludes that the ambient space is $\mathbb{R}^2$.

Let $\gamma$ be a closed critical curve in $\mathbb{H}^2(\rho)$. If $\varrho>0$ is the period of the   curvature $\kappa(s)$, from \eqref{psi(s)} we have that 
$$
\int_0^\varrho \frac{x(t)}{d-\rho\left(1-\log x(t)\right)^2x^2(t)}\,dt=0. $$
  This is impossible because the denominator is positive, proving  the result in  $\mathbb{H}^2(\rho)$.

Suppose that the ambient space is $\mathbb{S}^2(\rho)$ with   $\rho>4\mu^2$. In order to obtain the values of the constant of integration $d$ for which there exist periodic solutions, notice that orbits in this case cut the $x$-axis in either one or three points, in the latter, we obtain the closed ones. Indeed, we have that the function $F(x,0)=\mu^2x^2+\rho\left(1-\log x\right)^2x^2$  increases in the interval $\left(0,x_{-}\right)$, decreases as $x$ moves from $x_{-}$ to $x_{+}$ and, finally,   increases  if $x>x_+$.  We deduce that  $F(x,0)$ reaches a local maximum at $x_{-}$ and a local minimum at $x_{+}$. Therefore, there are exactly three cuts if and only if $d=F(x,0)$ for $x\in\left(x_{-},x_{+}\right)$. Since the function  $F(x,0)$ is decreasing in the interval $(x_{-},x_{+})$, a periodic solution appears if and only if
$$F(x_{+},0)=\rho x_{+}^2\log x_{-}<d<\rho x_{-}^2\log x_{+}=F(x_{-},0).$$
Again by \eqref{gg} and  \eqref{psi(s)}, if $\gamma$ is a closed curve, then the number $\sqrt{\rho d}\psi(\varrho)$ is a rational multiple of $2\pi$, obtaining \eqref{I(d)}.  
\end{proof}

\section{Classification of the extremal curves}\label{sec6}

Once obtained in Section \ref{sec5} that the generating curves of rotational surfaces of constant astigmatism in space forms are parametrized by \eqref{gg}, we   give the classification of these curves according to their shapes. To this end, in each one of the three following subsections we will analyze the phase plane \eqref{phase} in each   $2$-space form giving  a systematics on the names of all possible shapes.  By Propositions \ref{sym} and \ref{cuts},  an extremal curve $\gamma\subset {\mathbb M}^2(\rho)$ is symmetric about the geodesic $\alpha=\Pi_{13}\cap {\mathbb M}^2(\rho)$ and $\gamma$ may meet the geodesic $\beta=\Pi_{23}\cap {\mathbb M}^2(\rho)$ at two points, possibly singular points. We begin by summarizing the   geometric description of the critical curves of $\mathbf{\Theta}_\mu$ in ${\mathbb M}^2(\rho)$.

A first set of shapes are the curves obtained by   Lilienthal for the rotational surfaces  of constant astigmatism in  $\mathbb{R}^3$ (\cite{Lilienthal}). Similar shapes will appear in $\mathbb{S}^2(\rho)$ and $\mathbb{H}^2(\rho)$ and we will call them {\it Lilienthal's type curves}. In some of these shapes, the curves present peaks, that is, points where the curve is not defined and correspond with the value $x=1$ in \eqref{F(x,y)}. Moreover, by Proposition \ref{cuts}, those curves that tend to meet the geodesic $\beta$ at the end points, the intersection must occur orthogonally. The Lilienthal's type curves are the following (see figure \ref{curvasr21} for the shapes in Euclidean space):
\begin{enumerate}
\item \emph{Arch type curves}. Concave graphs of a function defined in a bounded interval  of the geodesic $\beta$. This function has a maximum   where the curve meets the symmetry axis.
\item \emph{Fishtail type curves}. Non simple curves with one intersection point on the symmetry axis. They have exactly two peaks  and a local minimum between the two peaks on the symmetry axis.
\item \emph{Deltoid type curves}. Simple curves having the shape of a deltoid  and with two peaks  and  one vertex at the intersection point between the geodesics $\alpha$ and $\beta$.
\item \emph{Bridge type curves}. Simple curves having the shape of a bridge. The towers bend away from each other and finish at exactly two peaks. Moreover, the cable joining the towers reach a local minimum on the symmetry axis.
\end{enumerate}

We turn now to those extremal curves that do not belong to Lilienthal's type family. These shapes  only appear in   $\mathbb{S}^2(\rho)$ when $\rho>4\mu^2$ and in  $\mathbb{H}^2(\rho)$: see figures \ref{curvass22especial} and \ref{curvash22} respectively. We give the next definitions.  
\begin{enumerate}
\item \emph{Anti-deltoid type curves}. Non simple curves having three vertices. These curves only appear in $\mathbb{S}^2(\rho)$ for $\rho>4\mu^2$. Moreover, one  of the vertices is located at the intersection point between $\alpha$ and $\beta$. The opposite segment to this vertex gives one turn around the north pole before closing.
\item \emph{Anti-fishtail type curves}. This case only appears in $\mathbb{S}^2(\rho)$. Non simple curves having the fishtail type shape but not between the two peaks the curve turns around the north pole before closing. 
\item \emph{Anti-arch type curves}. This case only appears in $\mathbb{S}^2(\rho)$.  Non simple curves having the arc type shape but now between the two peaks the curve turns around the north pole before closing. 
\item \emph{Anti-bridge type curves}. This case only appears in $\mathbb{S}^2(\rho)$. Simple curves having the shape of a bridge where the cable goes outside towers. Towers bend towards each other finishing exactly with two peaks   and the cable gives more than half turn around the north pole.
\item \emph{Cross type curves}. This case only appears in $\mathbb{S}^2(\rho)$. Non simple curves having some intersection points in the symmetry axis $\alpha$. They have two peaks.  After $x=1$, the curves give as many turns  around the north pole as needed before meeting the symmetry peak so they may have more than one self-intersection points. 
\item \emph{Braid type curves}. Complete curves with periodic curvature that roll up around a circle giving turns around the pole of the parametrization. When the curvature of the critical curve is not constant, this case only appears in $\mathbb{S}^2(\rho)$ for $\rho>4\mu^2$ and the curve may close up if the integral in \eqref{I(d)} is a multiple of $2\pi$. However, Euclidean circles are also included here as limit cases which can also appear in $\mathbb{H}^2(\rho)$.
\item \emph{Hypercycle type curves}. Concave graphs of a function defined on the entire geodesic $\beta$ and going further from it. This function has a minimum precisely where the curve meets the symmetry axis $\alpha$. These curves only appear in $\mathbb{H}^2(\rho)$. Hypercycles are included here as limit cases which   appear in $\mathbb{H}^2(\rho)$.
\item \emph{Anchor type curves}. This case only appears in $\mathbb{H}^2(\rho)$. Simple curves having two disjoint symmetric components with respect to the geodesic $\alpha$. Each component begins at the geodesic $\beta$ and cross it one more time before they tend to $\alpha$. It is the only one where $x$ does not have critical points.
 
\end{enumerate}
For the non-Lilienthal's type curves, the critical curves tend to meet the geodesic $\beta$ orthogonally at the end points, except for those described in the items 6 and 7 above. 

\subsection{The Euclidean plane $\mathbb{R}^2$}

After a dilation (Proposition \ref{unique}), we assume     that the index energy is $\mu=1$ and thus, the extremal curves are parametrized by  the constant of integration $d$ in \eqref{F(x,y)}.   In fact,   all these curves are determined by the relation between the  integration constant $d$ and the values $\rho+\mu^2=1$ and 
$d_*={\rm e}^2/4$.

\begin{theorem}\label{curvesR2} The critical curves of $\mathbf{\Theta}_1$ in $\mathbb{R}^2$ form a one-parameter family of curves depending on the constant of integration $d>0$ (see figure \ref{curvasr21}).
\begin{enumerate}
\item Case $d\leq 1$. The curves are of arch type.
\item Case $d\in\left(1,{\rm e}^2/4\right)$. The curves are of fishtail type.
\item Case $d={\rm e}^2/4$. The curves are of deltoid type.
\item Case $d>{\rm e}^2/4$. The curves are of bridge type. If $d\in({\rm e}^2/4,{\rm e}^2)$, the minimum of the cable is located in the positive part of the $\alpha$ axis; if $d={\rm e}^2$, the point $(0,0)$ represents this minimum, and; if $d>{\rm e}^2$, this minimum has negative component: in this case, the curves cut twice more times the $\beta$ axis making an angle $$\theta_d=\arccos\left(-\frac{{\rm e}}{\sqrt{d}}\right)$$ 
that varies from $\pi$ to $\pi/2$ as $d$ increases.
\end{enumerate}
\end{theorem} 

\begin{proof} 
 From \eqref{gg} and \eqref{psi(x)}, a critical curve $\gamma$ of  $\mathbf{\Theta}_1$ is parametrized   by
\begin{equation}
\gamma(x)=\frac{1}{\sqrt{d}}\left((1-\log x)x,\int\frac{x\log x}{\sqrt{d-x^2}}\,dx\right). \nonumber
\end{equation}
The integral can be computed obtaining 
\begin{equation}
\gamma(x)=\frac{1}{\sqrt{d}}\left((1-\log x)x, \sqrt{d-x^2}(\log x-1)+\sqrt{d}\log\left(\frac{\sqrt{d}+\sqrt{d-x^2}}{x}\right)\right). \label{paramx}
\end{equation}
From \eqref{F(x,y)}, the function $F(x,0)=x^2$ is increasing for any $x>0$. This implies that, for any fixed $d>0$, the associated orbit $F(x,y)=d$  intersects the $x$-axis. Moreover,   this happens precisely at $x_0=\sqrt{d}$, that is, at the point of the orbit where the maximum value of $x$ is reached. 

We will classify and describe the critical curves only for values of $x$ in the interval  $\left(0,x_0\right)$, that is, the half of the  curves because the other half is obtained  by symmetry (Proposition \ref{sym}). 

\begin{enumerate}
\item Case $x_0=\sqrt{d}\leq 1$. The function $\psi(x)$ monotonically increases until $\psi$ reaches its maximum at $x_0$, namely, $\psi(x_0)=0$. At the same time, the function $(1-\log x)x$ also increases from $x=0$ to $x=x_0$. This implies that $\gamma$ is  of arch type.
\item Case $x_0=\sqrt{d}\in\left(1,{\rm e}/2\right)$. The curve $\gamma$ is   defined by two parts. If $x\in (0,1)$, the function  $\psi(x)$ increases from a negative number ($\lim_{x\rightarrow 0}\psi(x)<0$) to a positive one ($\lim_{x\rightarrow 1}\psi(x)>0$). Therefore, there is a point where $\gamma$ intersects the geodesic $\alpha$ and, by symmetry, it represents a self-intersection point. This self-intersection point occurs far from the geodesic $\beta$ because $\lim_{x\rightarrow 0}\psi(x)<0$. On the other hand, the second part of $\gamma$ corresponds with $x>1$, where  the function $\psi(x)$ decreases from $\lim_{x\rightarrow 1}\psi(x)>0$ until the value   $\psi(x_0)=0$. Since   that the function $(1-\log x)x$ is decreasing if $x>1$, the curve is   of fishtail type.
\item Case $x_0=\sqrt{d}={\rm e}/2$. The behavior of $\gamma$ is as in previous case but now  the self-intersection point appears when $x\rightarrow 0$, that is, precisely at the intersection point with $\beta$. Thus the curve is of deltoid type.
\item Case $x_0=\sqrt{d}>{\rm e}/2$. The curve is again defined in two parts as in the items $2$ and $3$. There are no self-intersection points because $\lim_{x\rightarrow 0}\psi(x)>0$ and $\psi(x)$ increases in the interval $(0,1)$. For this value of $x_0$, and for $x>1$, the function $\psi(x)$ decreases until $\psi(x_0)=0$. Now the value   $(1-\log x_0)x_0$ may be positive, zero or negative depending if $x_0<{\rm e}$, $x_0={\rm e}$ or $x_0>{\rm e}$, respectively. In each case, we obtain bridge type curves of the three different possible cases. Using $y=x_s$, the tangent vector field $T$ of $\gamma$  can be computed in terms of the arc-length parameter obtaining
\begin{equation}
T(s)= \gamma'(s)=-\frac{1}{\sqrt{d}}\left(y\log x,x\right)=-\frac{1}{\sqrt{d}}\left(\sqrt{ d-x^2},x\right),\nonumber
\end{equation}
where in the last equality we use equation \eqref{F(x,y)} for $\rho=0$. The angle $\theta(s)$  between $T(s)$ and any parallel curve to $\beta$ satisfies
$$\cos\theta(s)= -\frac{x(s)}{\sqrt{d}}.$$
In particular, regular points representing the intersection of $\gamma$ with the axis $\beta$ appear when $x={\rm e}$.
\end{enumerate}
\end{proof}

\begin{figure}[hbtp]
	\begin{centering}{\includegraphics[angle=90,width=2.45cm,height=2.45cm]{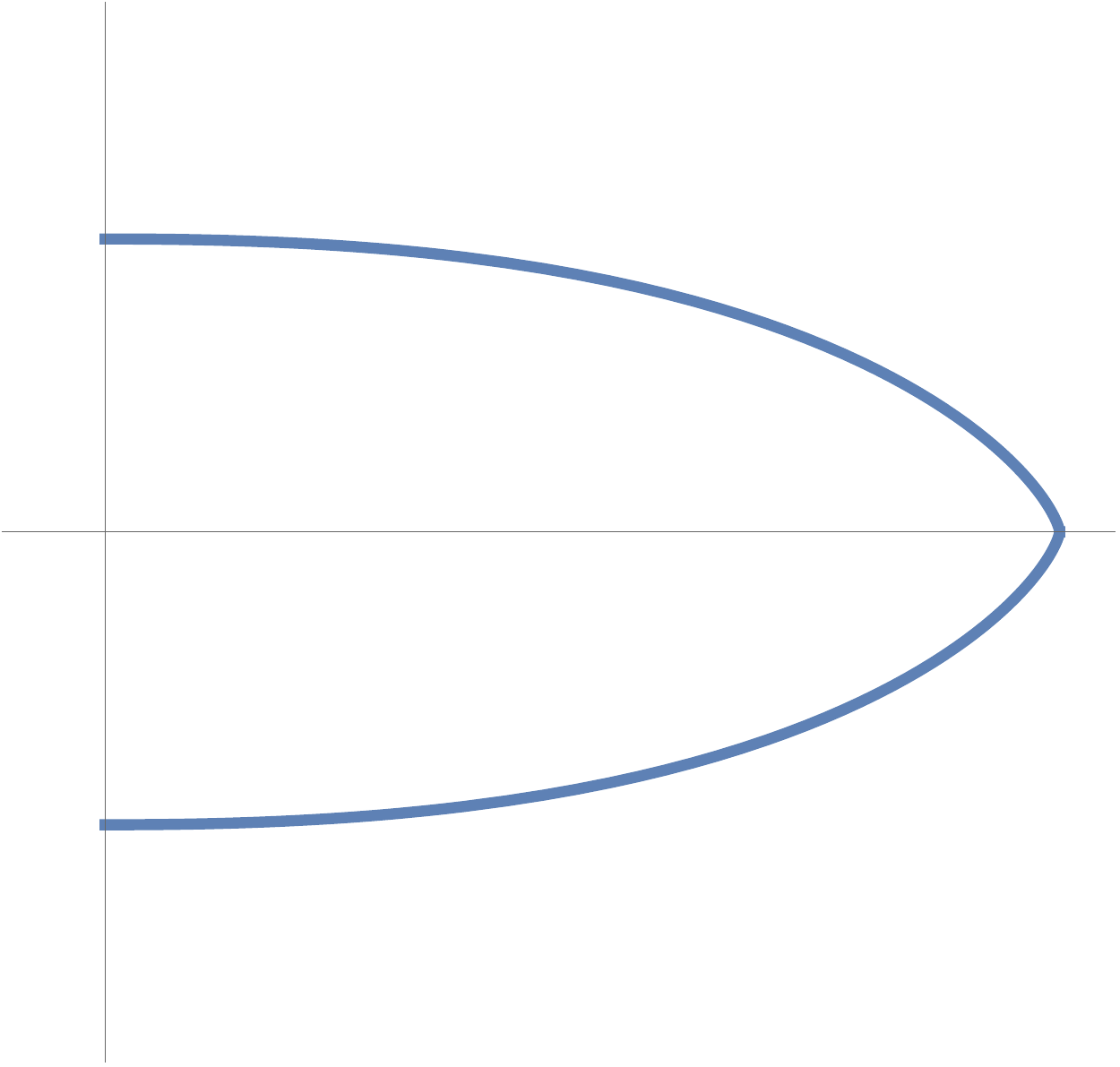}
	\includegraphics[angle=90,width=2.45cm,height=2.45cm]{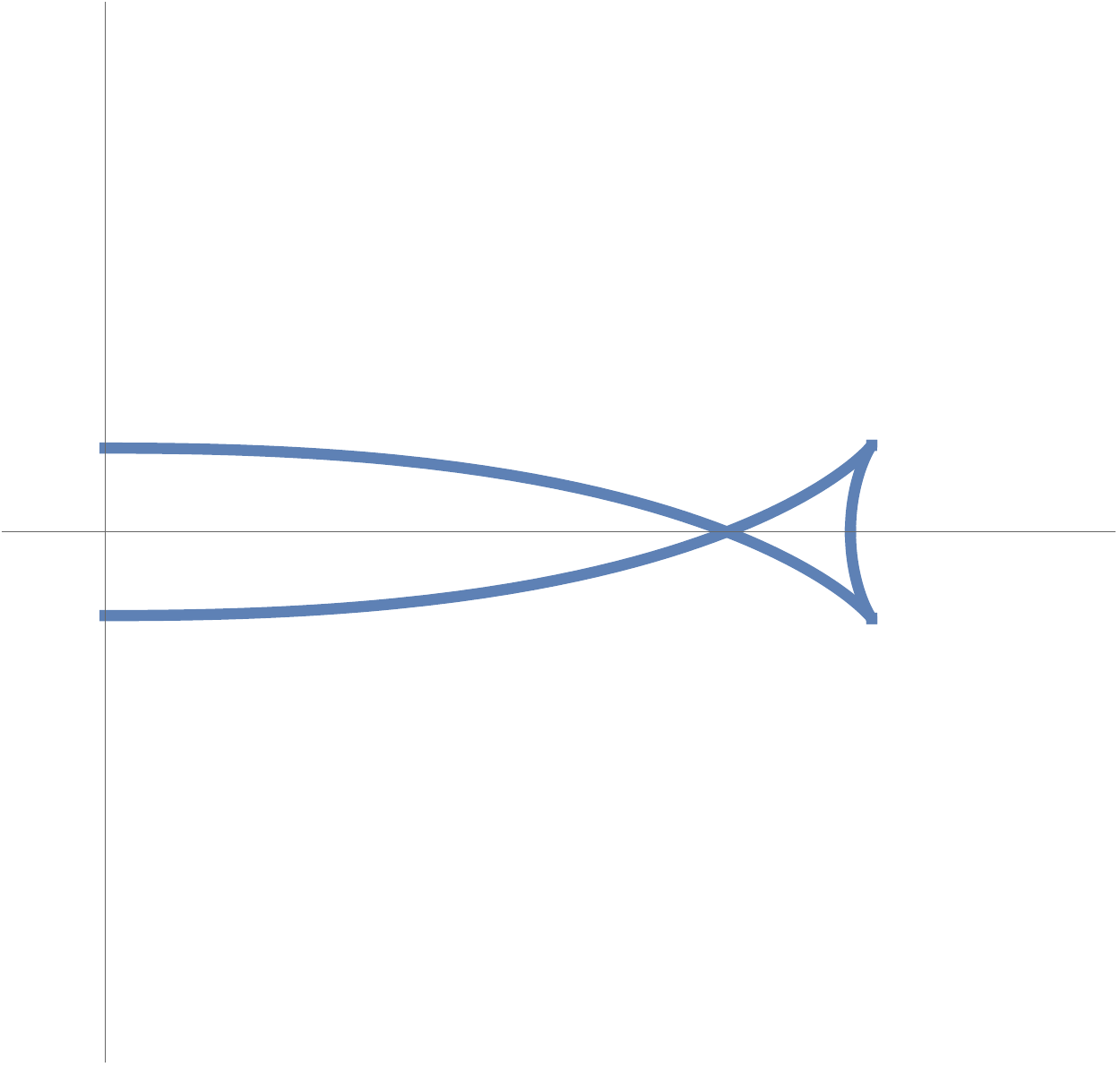}
	\includegraphics[angle=90,width=2.45cm,height=2.45cm]{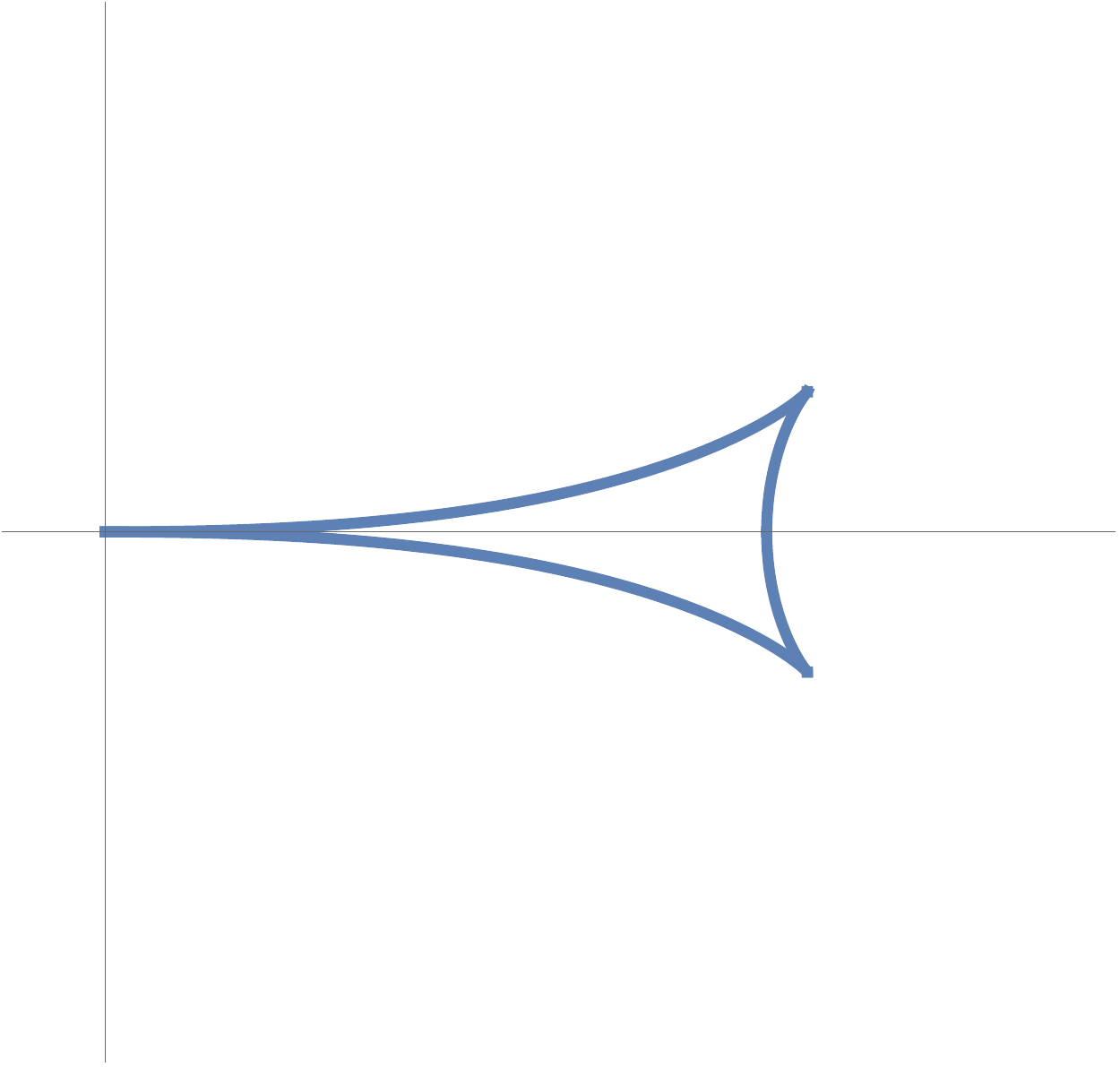}
	\includegraphics[angle=90,width=2.45cm,height=2.45cm]{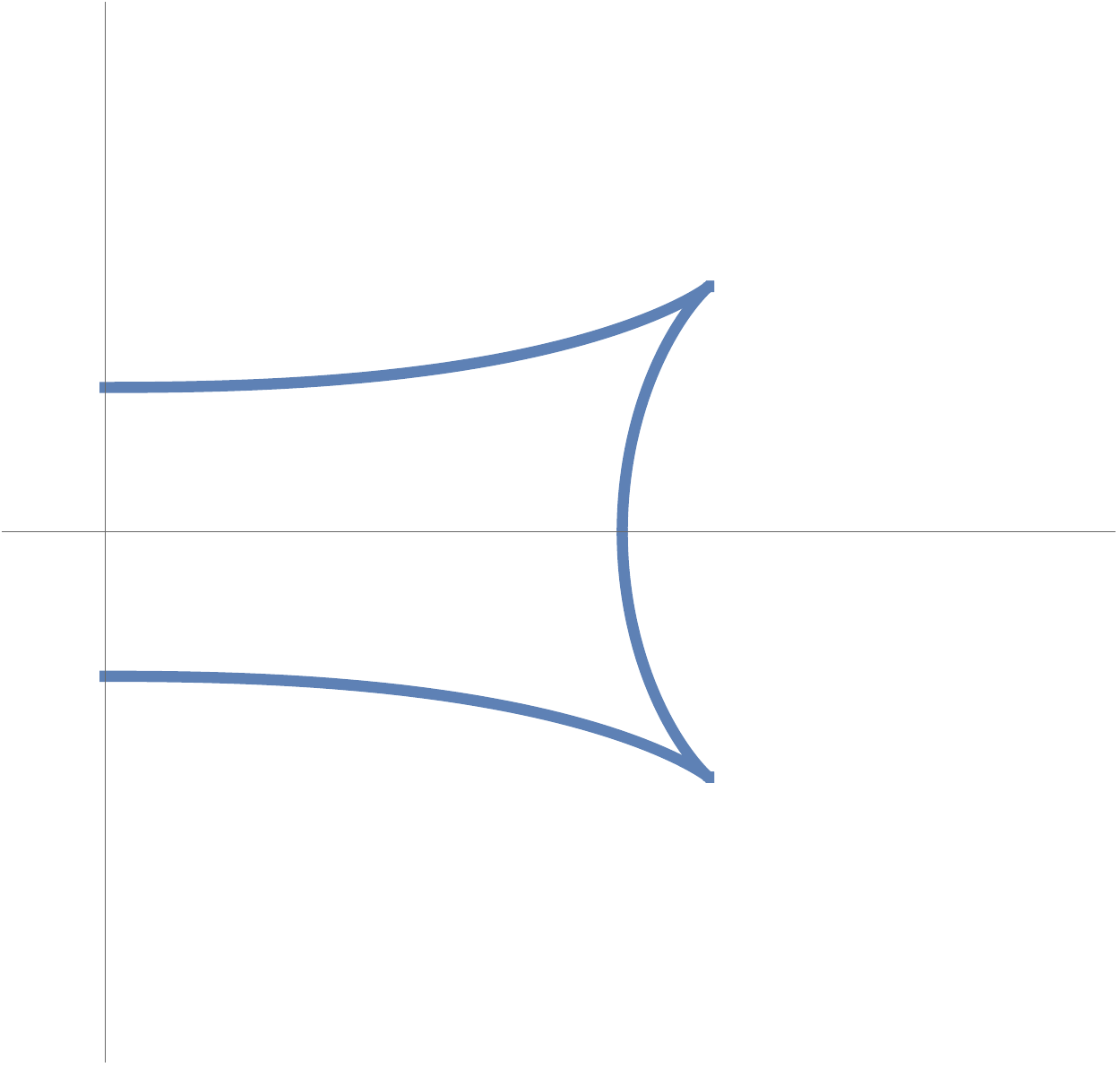}
	\includegraphics[angle=90,width=2.45cm,height=2.45cm]{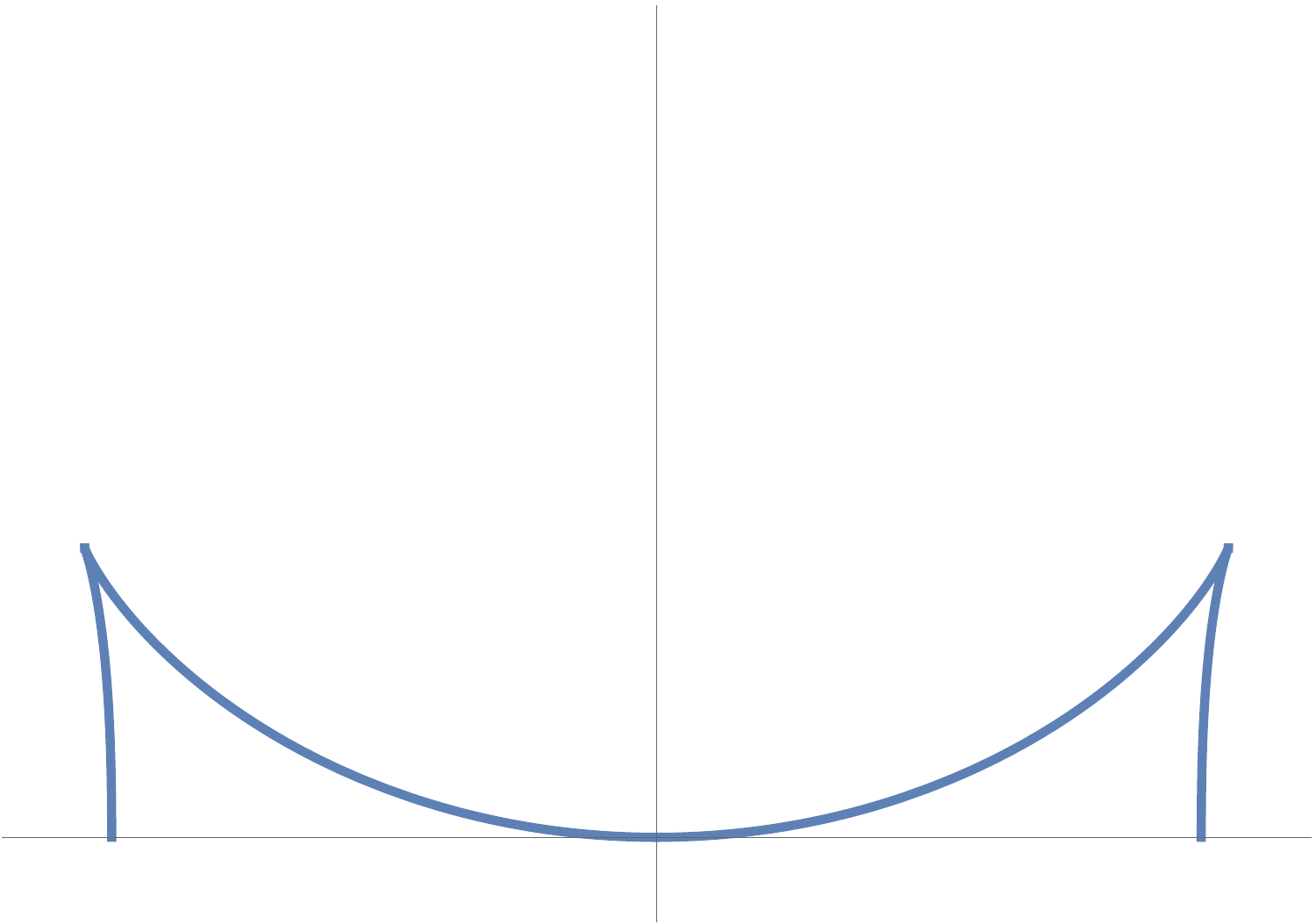}
	\includegraphics[angle=90,width=2.45cm,height=2.45cm]{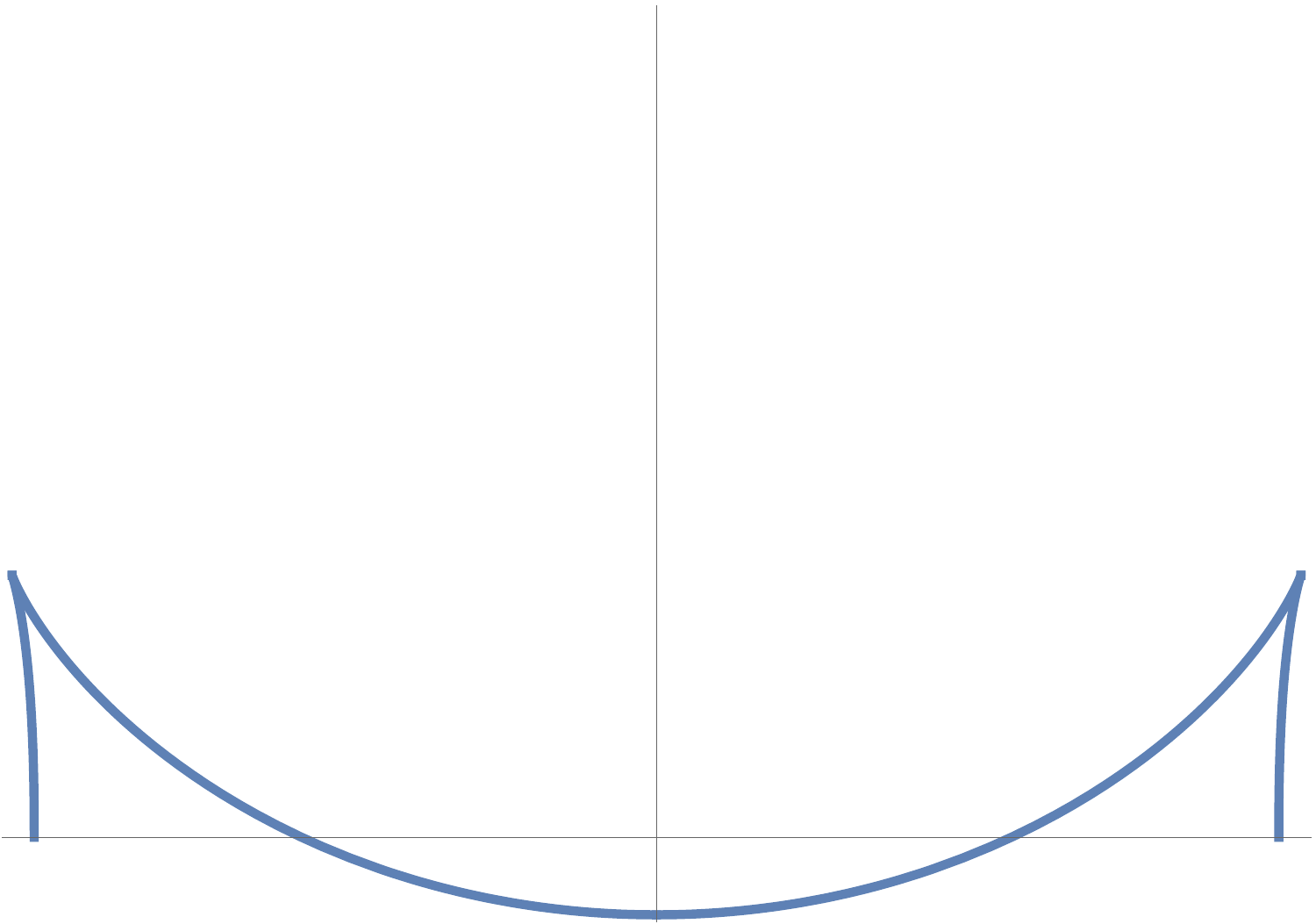}}
		\caption{Critical curves of $\mathbf{\Theta}_1$ in the Euclidean plane depending on the constant $d>0$. The geodesic $\beta$ is represented by the horizontal axis, whereas $\alpha$ is the vertical one. From left to right: arch type ($d=1$), fishtail type ($d=1.55$), deltoid type ($d={\rm e}^2/4$) and the three different possible cases of bridge type ($d=2.5$, $d={\rm e}^2$ and $d=9$, respectively)}
		\label{curvasr21}
	\end{centering}
\end{figure}
\subsection{The sphere $\mathbb{S}^2(\rho)$}

Consider  the  $2$-sphere $\mathbb{S}^2(\rho)$. By  Proposition \ref{unique}, it is enough to consider that the energy index $\mu$ is positive. We  know by Section \ref{sec5} that if $\rho<4\mu^2$ there are no singular points; if $\rho=4\mu^2$, there is only one  degenerate singular point $P$, and; if $\rho>4\mu^2$, there are two singular points, $P_{+}$ and $P_{-}$ with $x_{-}<x_{+}$, representing a centre and a saddle point, respectively.

In the following two theorems we sum up the geometric description of critical curves of $\mathbf{\Theta}_\mu$  depending on the relation between the constant sectional curvature $\rho$, the positive energy index $\mu$ and the value $d_*$ defined  after equation \eqref{psi(x)}. The first result discusses the case $\rho\leq 4\mu^2$. As we will see, all shapes are of Lilienthal's type: see figure    \ref{curvass21}.

\begin{theorem}\label{curvesS21} $($Case $\rho\leq 4 \mu^2)$ The critical curves of $\mathbf{\Theta}_\mu$ in $\mathbb{S}^2(\rho)$ for $\rho\leq 4 \mu^2$ represent a one-parameter family depending on the constant of integration $d>0$ in \eqref{F(x,y)}.
\begin{enumerate}
\item Case $d\leq \rho+\mu^2$. The curves are of arch type.
\item Case $d\in\left(\rho+\mu^2, d_*\right)$. The curves are of fishtail type.
\item Case $d= d_*$. The curves are of deltoid type.
\item Case $d> d_*$. The curves are of bridge type. If $d\in\left(d_*,\mu^2{\rm e}^2\right)$, the minimum of the cable is located in the upper halfsphere; if $d=\mu^2{\rm e}^2$, the minimum is the point where the equator $\beta$ intersects the symmetry axis; and, finally, if $d>\mu^2{\rm e}^2$, the minimum is located in the lower halfsphere, what means that the curves meet  twice the equator $\beta$. 
\end{enumerate}
\end{theorem}

\begin{proof} Let $\gamma\subset\mathbb{S}^2(\rho)$ be a critical curve of $\mathbf{\Theta}_\mu$ for $\mu>0$ such that $\rho\leq 4 \mu^2$.   Now we have   $F(x,0)=x^2(\mu^2+\rho(1-\log x)^2)$. Then the derivative of $F(x,0)$ with respect to $x$ is $2x(\mu^2-\rho\log x(1-\log x))$ and it is positive because  $\rho\leq 4 \mu^2$. This implies that $F(x,0)$  monotonically increases from $0$ to $+\infty$.  Thus,  for any fixed $d>0$, the associated orbit $F(x,y)=d$   always intersects the $x$-axis of the phase plane at exactly one point  $x_0$, which is the maximum value of $x$ in that orbit.

As in the Euclidean case, we are going to argue only for values $x\in\left(0,x_0\right)$, since the other half part of $\gamma$  can be completed by symmetry. We have the following cases:
\begin{enumerate}
\item Case $x_0\leq 1$. Then necessarily $d\leq \rho+\mu^2$. The function $\psi(x)$   monotonically increases until $x=x_0$, where $\psi(x_0)=0$ and $\gamma$ meets the symmetry axis. Moreover, the function $(1-\log x)x$ also increases and consequently,  $\gamma$ is a curve of arch type.
\item Case $x_0>1$. Then $d>\rho+\mu^2$ and we have different cases. If $d\in\left(\rho+\mu^2,d_*\right)$ then $\lim_{x\rightarrow 0} \psi(x)<0$ and $\psi(x)$ increases until $x= 1$. Moreover, since $\lim_{x\rightarrow 1}\psi(x)>0$,   there exists a self-intersection point for some $x\in\left(0,1\right)$. For $x>1$, the function $\psi(x)$   decreases until reaching the value $x_0$, where $\psi(x_0)=0$. At the same time, the function $(1-\log x)x$ increases if $x\in\left(0,1\right)$ and decreases for $x\in\left(1,x_0\right)$ giving rise to critical curves of fishtail type.
\item Case $d=d_*$. By definition of $d_*$,  we have that $\lim_{x\rightarrow 0}\psi(x)=0$ and arguing as above we conclude that $\gamma$ is a curve of deltoid type.
\item Case $d>d_*$. Then $\lim_{x\rightarrow 0}\psi(x)>0$ and $\psi(x)$ increases for $x\in\left(0,1\right)$. For these values of $x$, the function $(1-\log x)x$ also increases. However, in the interval $(1,x_0)$, the functions $\psi(x)$ and $(1-\log x)x$ decrease. Exactly, at $x=x_0$, the function $\psi(x)$ vanishes which means that $\gamma$ meets the geodesic $\alpha$. Moreover, if $d<\mu^2{\rm e}^2$, this contact point is located in the upper halfsphere; if $d=\mu^2{\rm e}^2$, $\gamma(x_0)$ is the intersection point between $\alpha$ and $\beta$, and;  for $d>\mu^2{\rm e}^2$, the point $\gamma(x_0)$ is located in the lower halfsphere. Thus there appear the   three different possibilities of bridge type curves.\end{enumerate}\end{proof}

\begin{figure}[hbtp]
	\begin{centering}{\includegraphics[width=2.45cm,height=2.45cm]{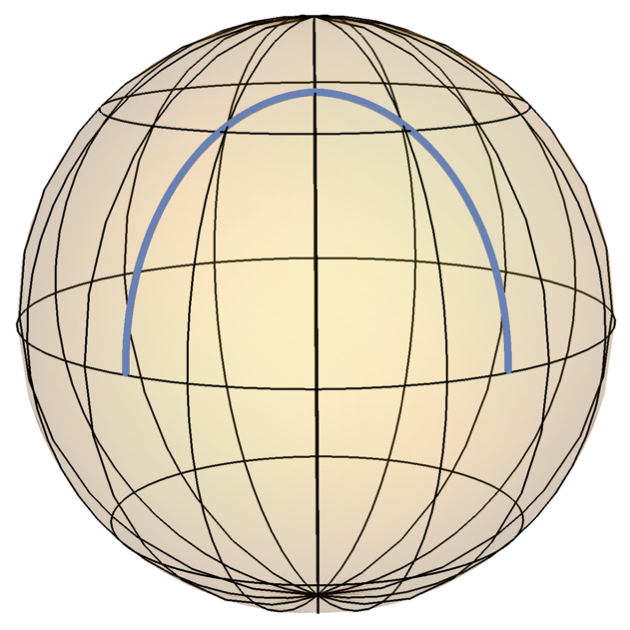}
	\includegraphics[width=0.15\textwidth]{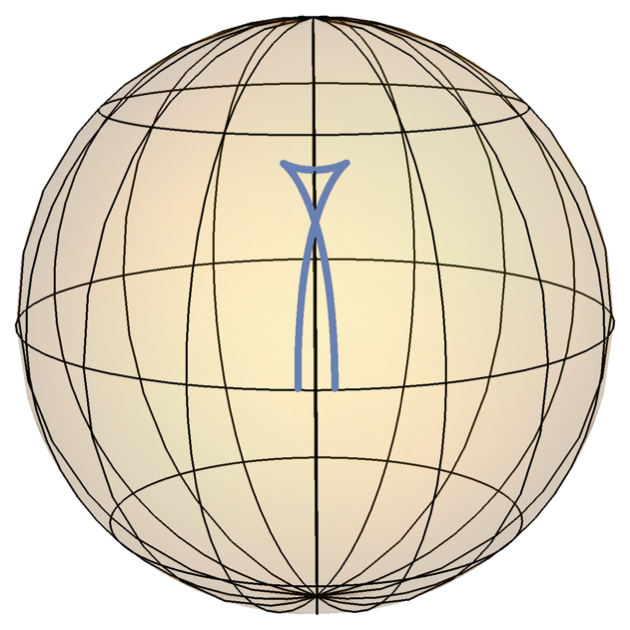}
	\includegraphics[width=0.15\textwidth]{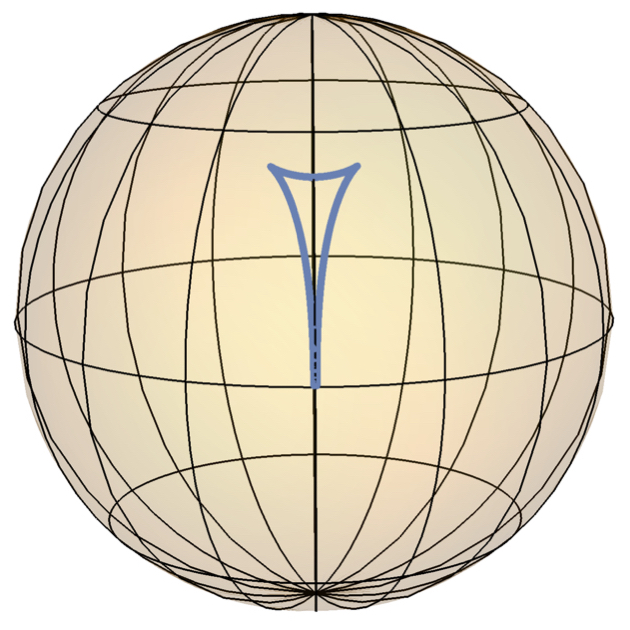}
	\includegraphics[width=0.15\textwidth]{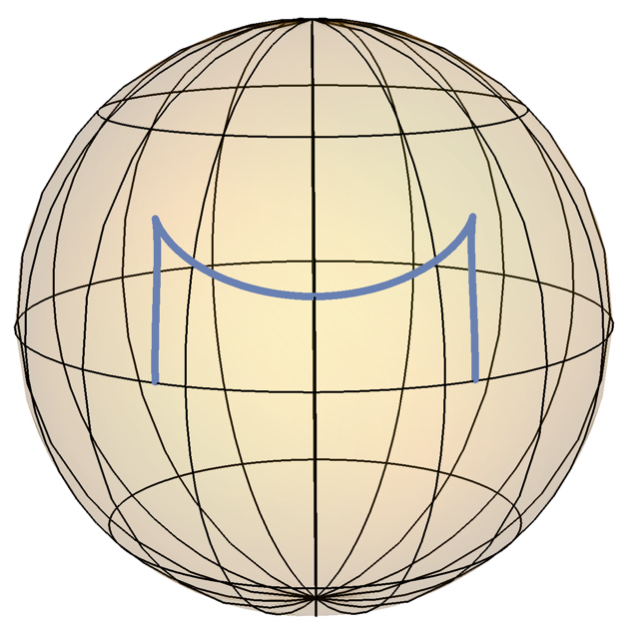}
	\includegraphics[width=0.15\textwidth]{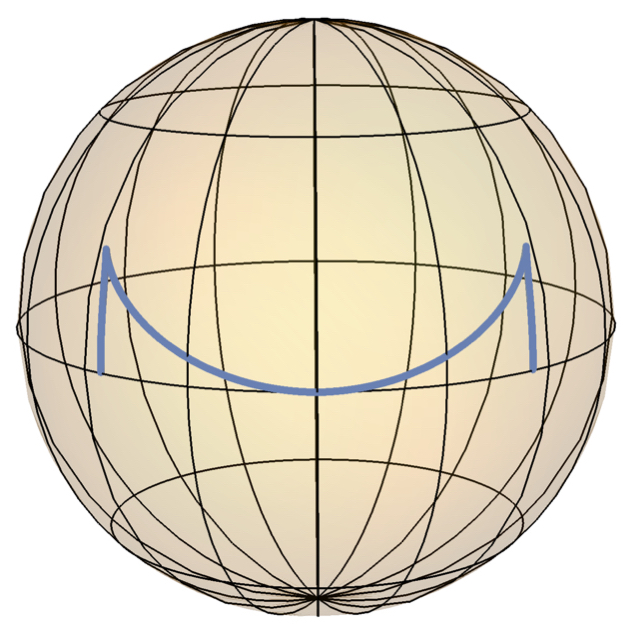}
	\includegraphics[width=0.15\textwidth]{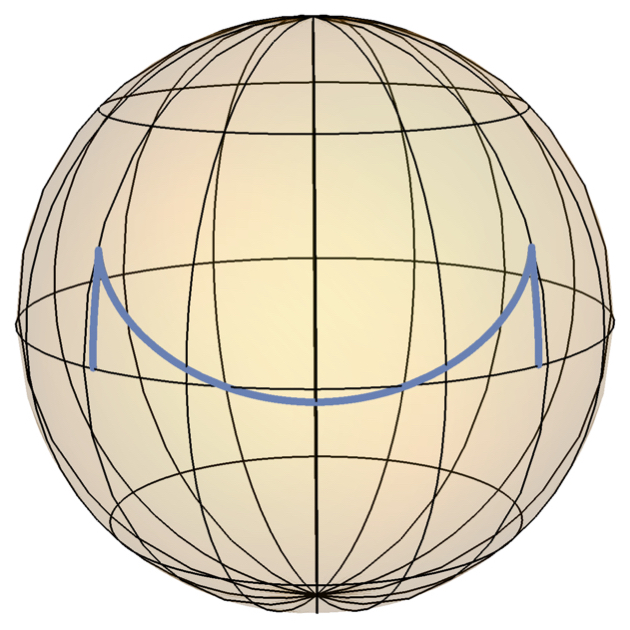}}
		\caption{Critical curves of $\mathbf{\Theta}_\mu$ in $\mathbb{S}^2(\rho)$ for $\rho<4\mu^2$. Here $\rho=\mu=1$. From left to right: arch type ($d=1.3$), fishtail type ($d=2.5$), deltoid type ($d=d_*\simeq 2.64$) and three different cases of bridge type ($d=4.5$, $d=\mu^2{\rm e}^2$ and $d=8$, respectively)}
		\label{curvass21}
	\end{centering}
\end{figure}

 We now study the case   $\rho>4\mu^2$, which has not a counterpart in Euclidean plane, being some of these  shapes of non-Lilienthal's type curves: see figure \ref{curvass22especial}.

\begin{theorem}\label{curvesS22} $($Case $\rho>4\mu^2)$ The critical curves of $\mathbf{\Theta}_\mu$ in $\mathbb{S}^2(\rho)$ for $\rho >4\mu^2$ are of arch type, fishtail type, deltoid type, bridge type, cross type, anti-deltoid type, anti-fishtail type, anti-arch type and anti-bridge type. Moreover, we also have braid type curves whenever $d\in\left(\rho x_{+}^2\log x_{-},\rho x_{-}^2\log x_{+}\right)$.
\end{theorem}

\begin{proof} Assume that $\gamma$ is a critical curve of $\mathbf{\Theta}_\mu$ in $\mathbb{S}^2(\rho)$ for $\rho>4\mu^2>0$. Arguing as in the beginning of the proof of Theorem \ref{curvesS21}, in this case there are two values $x_{-}$ and $x_{+}$ with $x_{-}<x_{+}$, such that the function $F(x,0)$ increases if $x\in\left(0,x_{-}\right)$, then decreases if $x\in (x_{-},x_{+})$ and  finally increases if  $x>x_{+}$. Exactly, $x_{-}$ and $x_{+}$ are the values that appeared in the analysis of the phase plane in Section \ref{sec5}. Therefore, for each $d>0$,   the associated orbit $F(x,y)=d$   meets either once or thrice the $x$-axis (see Proposition \ref{closed} for details). 

Suppose that the orbit $F(x,y)=d$ has only one intersection point with the $x$-axis at   $x=x_0$. As   before, we observe  that $x_0$ represents the maximum value of $x$ in that orbit. This happens precisely for any $d\in\left(0,\rho x_{+}^2\log x_{-}\right)\cup\left(\rho x_{-}^2\log x_{+}, \infty\right)$. In this setting we have the same type of critical curves of Theorem \ref{curvesS21}, that is, arch type, fishtail type, deltoid type and bridge type curves.  Moreover, whenever $d$ is close enough to $\rho x_{-}^2\log x_{+}$, some different kind of critical curves may also appear for some prescribed values of $\mu$. We consider that $\lvert\lim_{x\rightarrow 0} \psi(x)\rvert$ is bigger than $\pi/4$. This means that the critical curve $\gamma$ begins and ends in the geodesic $\beta$ at a distance bigger than a half round. What is more, the value $\lvert\lim_{x\rightarrow 0}\psi(x)\rvert$ grows as $d$ decrease. Then, if the points where $\gamma$ is not defined (the vertices appearing with $x=1$) are far from each other we have, by a similar argument as in Theorem \ref{curvesS21}, anti-bridge type curves. Now, if the points corresponding with $x=1$ touch themselves we are dealing with anti-arch type curves. Then, for smaller values of $d$, the endpoints of $\gamma$ are close enough (after one turn in the sphere) so that the towers have a self-intersection point, that is, $\gamma$ is of anti-fishtail type. Following with this argument, there may exist a value of $d$ where the towers meet precisely at the intersection point between $\alpha$ and $\beta$. This case is of anti-deltoid type. Finally, for smaller values of $d$ the distance between the endpoints grows (linearly measured) and the critical curve may give as many turns around the north pole as needed. These curves have at least one self-intersection point in the symmetry axis $\alpha$ so they are of cross type.

The other possibility is that   the orbit has three points of intersection with the $x$-axis, that is, when $d\in\left(\rho x_{+}^2\log x_{-},\rho x_{-}^2\log x_{+}\right)$. If $x_0$ denotes the first one and, whenever $x<x_0$ we can argue as in the proof of Theorem \ref{curvesS21} obtaining that this part of $\gamma$ has been previously described. However, for these values of $d$ we also have another two intersection points which appear around the critical point $x_{+}$. Moreover, since $P_{+}=(x_{+},0)$ represents a centre, we have that the associated orbits are closed (Proposition \ref{closed}). Thus, the curve $\gamma$   has periodic curvature and $\gamma$  goes around the north pole being a curve of braid type. In general, the curve $\gamma$ is not   closed because it is necessary that the integral  in \eqref{I(d)} is a multiple of $2\pi$. 
\end{proof}

\begin{figure}[hbtp]
	\begin{centering}{
	\includegraphics[width=0.15\textwidth]{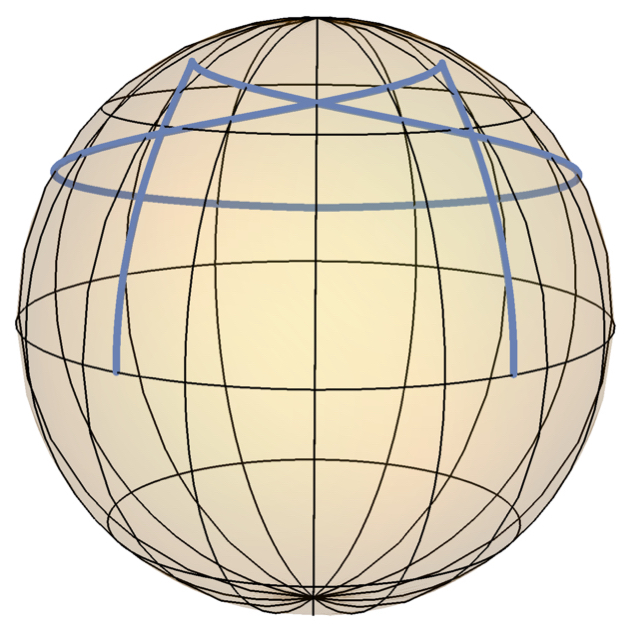}
	\includegraphics[width=0.15\textwidth]{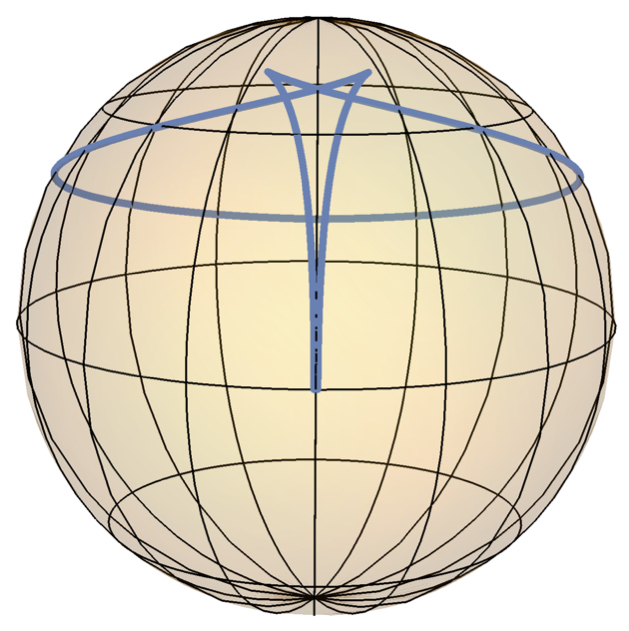}
	\includegraphics[width=0.15\textwidth]{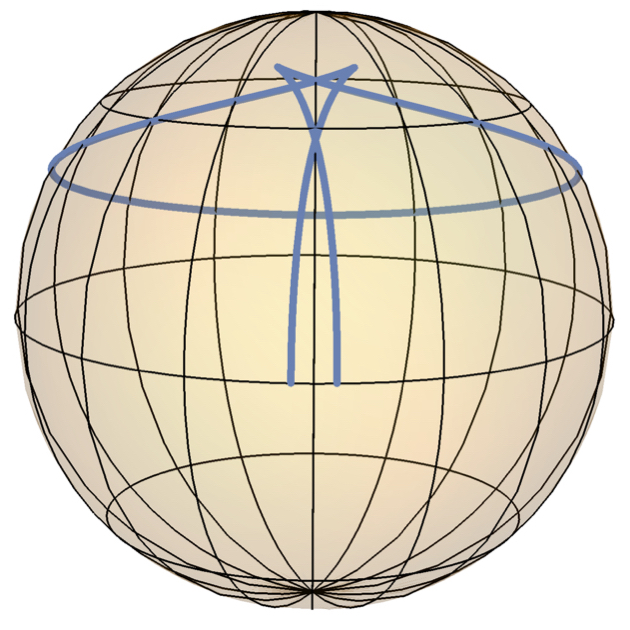}
	\includegraphics[width=0.15\textwidth]{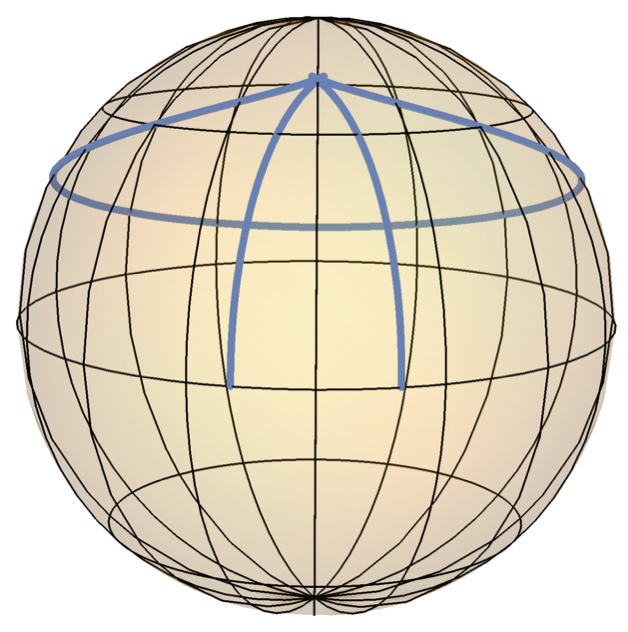}
	\includegraphics[width=0.15\textwidth]{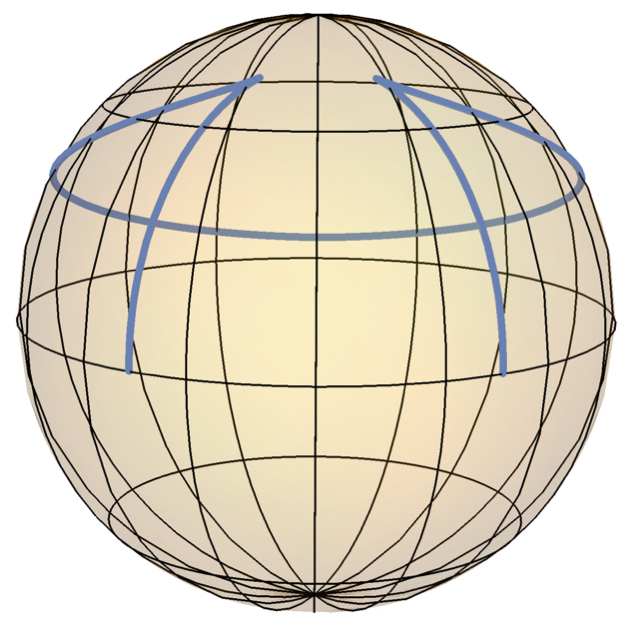}
	\includegraphics[width=0.15\textwidth]{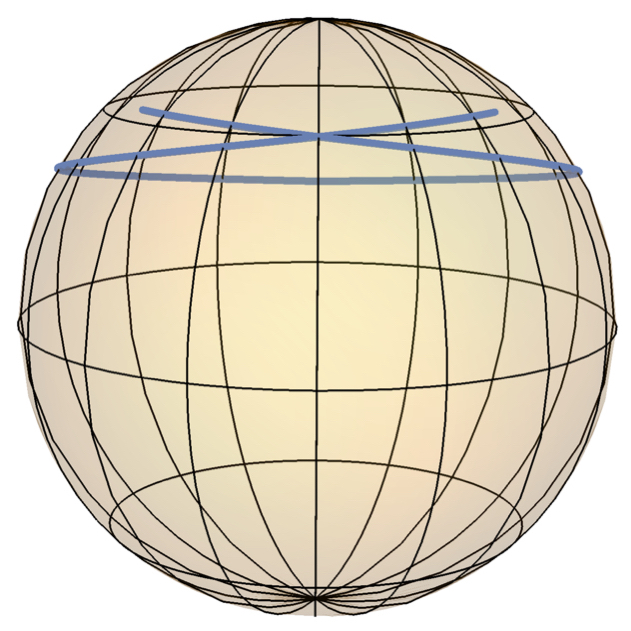}}
		\caption{Non-Lilienthal's type critical curves of $\mathbf{\Theta}_\mu$ in $\mathbb{S}^2(\rho)$   for $\rho> 4\mu^2$. Here $\rho=1$ and $\mu=0.45$. From left to right: cross type ($d=1.29$), anti-deltoid type ($d\simeq 1.32$), anti-fishtail type ($d=1.33$), anti-arch type ($d\simeq 1.35$), anti-bridge type ($d=1.4$) and one period of the curvature of a braid type ($d=1.23$)}
		\label{curvass22especial}
	\end{centering}
\end{figure}

\subsection{The hyperbolic plane $\mathbb{H}^2(\rho)$}

Recall that the classification of critical curves for $\mathbf{\Theta}_\mu$ in $\mathbb{H}^2(\rho)$ is done when the constant $d$ in \eqref{F(x,y)} is positive, that is,  the surfaces generated in   $\mathbb{H}^3(\rho)$ are of spherical type. We know from Section \ref{sec5} that in the hyperbolic plane $\mathbb{H}^2(\rho)$ there are two singular points   $P_{+}$ and $P_{-}$ in the phase plane,  $x_{+}<1<x_{-}$, representing a centre and an unstable saddle point, respectively: see figure \ref{orbitasr2}, right. The classification of the extremal curves is the following: see figures \ref{curvash21} and \ref{curvash22}.

\begin{theorem}\label{curvesH2}
Let $d>0$.  The critical curves of $\mathbf{\Theta}_\mu$ in $\mathbb{H}^2(\rho)$ for $d\in\left(0,\rho x_{-}^2\log x_{+}\right)$ are arch type curves, fishtail type curves, deltoid type curves, bridge type curves and hypercycle type curves. Moreover, whenever $d>\rho x_{-}^2\log x_{+}$ we also have anchor type curves. 
\end{theorem}

\begin{proof}
Let   $\gamma$ be a critical curve for $\mathbf{\Theta}_\mu$ in $\mathbb{H}^2(\rho)$ where $d>0$ is the constant of integration of \eqref{EL}. As in previous subsections, we take the function $F(x,0)$. Since now $\rho<0$, the function   $F(x,0)$ decreases while $x<x_{+}$. Because $P_{+}=(x_{+},0)$ represents a centre, then $F(x,0)$ reaches a local minimum at $x=x_{+}$.  In the interval  $\left(x_{+},x_{-}\right)$ the function $F(x,0)$ increases from the   number $F(x_{+},0)<0$ to the   $F(x_{-},0)=\rho x_{-}^2\log x_{+}>0$. By the same argument as before, the point  $x_{+}$ represents the point where $F(x,0)$ reaches its local maximum. Finally, for $x>x_{+}$, the function $F(x,0)$ is decreasing.

Therefore,  the orbit $F(x,y)=d$ meets the $x$-axis if and only if $d\leq \rho x_{-}^2\log x_{+}$ and, in this case, the orbit intersects the $x$-axis at exactly two points.   The first intersection point corresponds with a point $x=x_0$, and  $x_0<x_{-}$, while the second one appears for some $x=x_0'$, and $x_{0}'>x_{-}$. Since $P_{-}=\left(x_{-},0\right)$ is a saddle point, we get that the orbit $F(x,y)=d$ has two connected components, each of them corresponding with one intersection point.
 For the connected component corresponding with a point $x_0<x_{-}$, we can argue as in $\mathbb{R}^2$ and $\mathbb{S}^2(\rho)$   obtaining curves of arch type, of fishtail type, of deltoid type and of bridge type (Lilienthal's type curves). 

Let us consider now the connected component corresponding with the cut $x_{0}'>x_{-}$. For this case, we have that the functions $\psi(x)$ and $(1-\log x)x$ always decrease. Moreover,  $\lim_{x\rightarrow\infty}(1-\log x)x=-\infty$, which means that $\gamma$ is defined on the entire geodesic $\beta$. Therefore  $\gamma$ is a curve of hypercycle type.

Finally, we   consider that $d>\rho x_{-}^2\log x_{+}$. In this case, the associated orbits do not meet the $x$-axis and  are defined for all $x>0$ and $x\neq 1$. As $x\rightarrow 0$, the critical curve tends to meet $\beta$. Thus in the interval $x\in(0,1)$, the functions $\psi(x)$ and $(1-\log x)x$ increase. The point where $x=1$ represents a peak and $\gamma$ is not defined there. After that, both functions decrease. Since $\gamma$ is defined in $(0,1)\cup(1,\infty)$, there are points where $\gamma$ intersects $\beta$ (corresponding with $x={\rm e}$, see Proposition \ref{cuts}). Moreover, $\gamma$ never intersects the symmetry axis $\alpha$ since $\psi(x)>0$ for all possible values of $x$. However, $\lim_{x\rightarrow\infty} \psi(x)=0$, that is, $\gamma$ tends to the symmetry axis but it has two disjoint parts.. Thus, we get anchor type curves.
\end{proof}

\begin{figure}[hbtp]
	\begin{centering}{
	\includegraphics[width=0.15\textwidth]{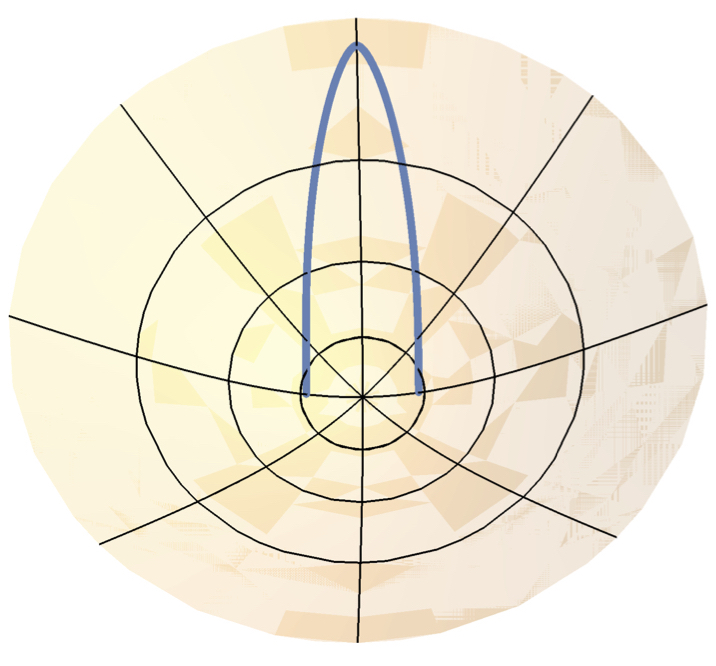}
	\includegraphics[width=0.15\textwidth]{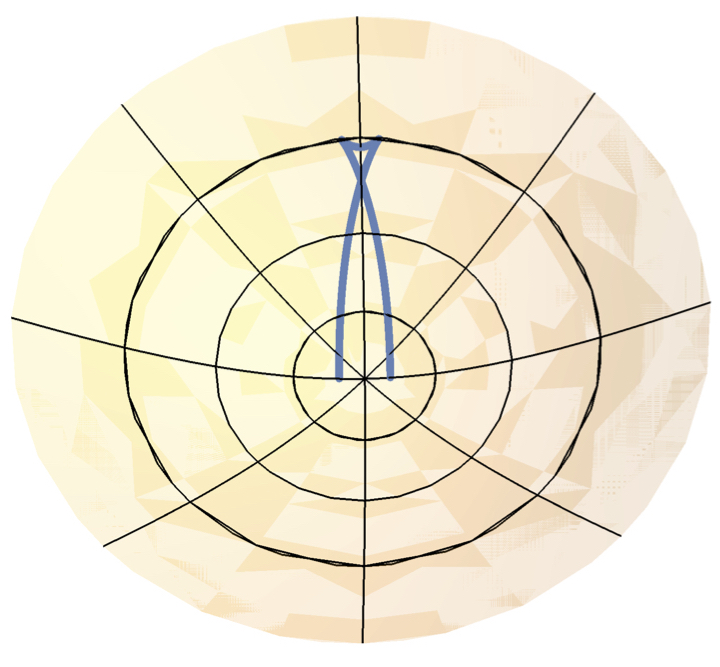}
	\includegraphics[width=0.15\textwidth]{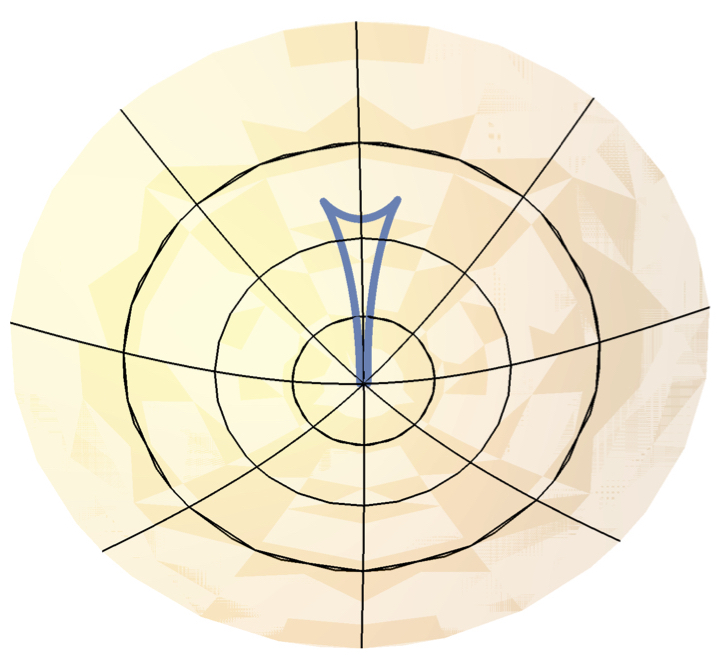}
	\includegraphics[width=0.15\textwidth]{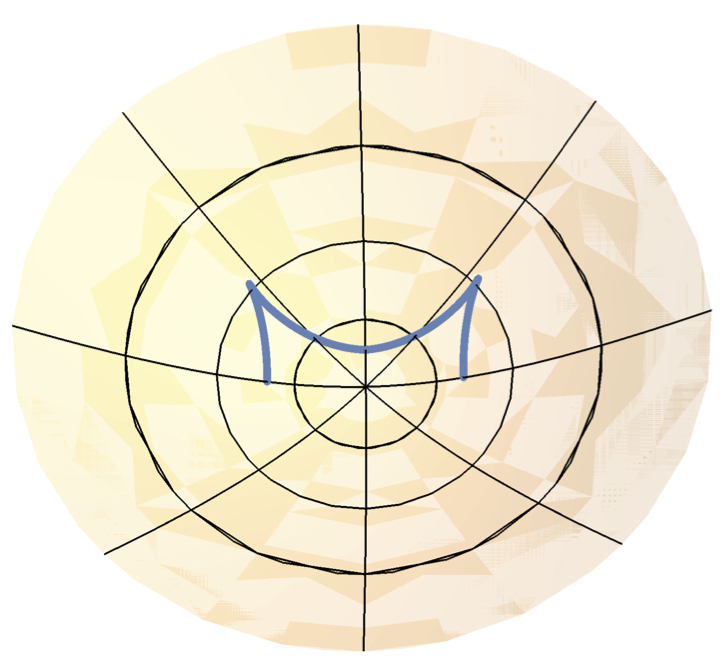}
	\includegraphics[width=0.15\textwidth]{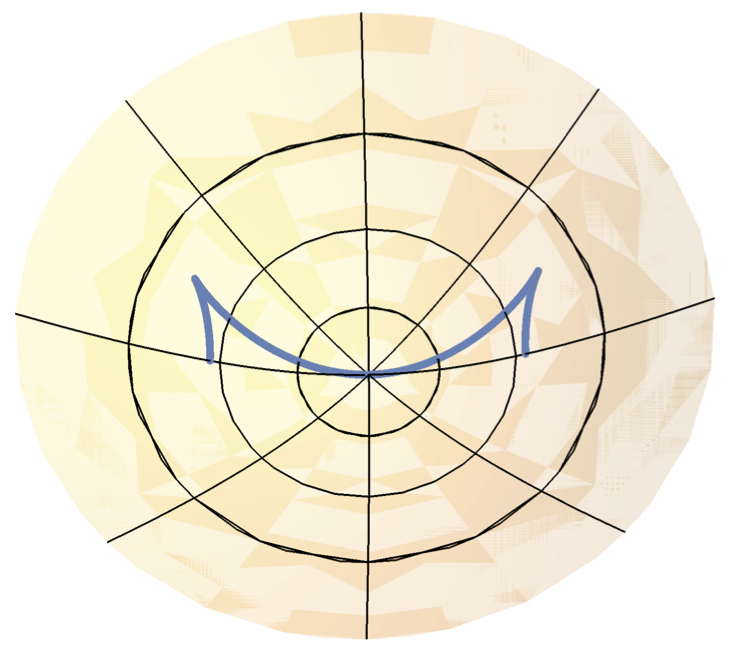}
	\includegraphics[width=0.15\textwidth]{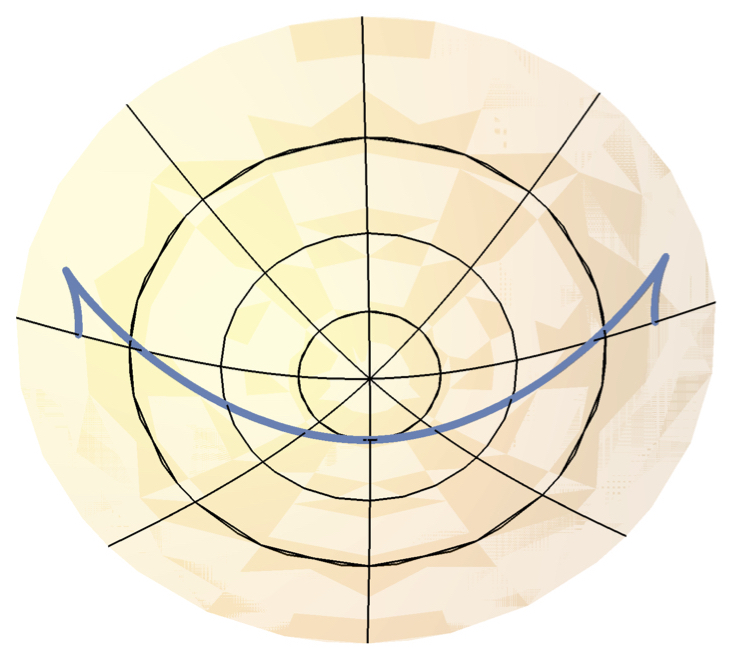}}
		\caption{Lilienthal's type critical curves of $\mathbf{\Theta}_\mu$ in $\mathbb{H}^2(\rho)$ depending on the constant $d>0$. Here $\rho=-1$ and $\mu=1.25$. From left to right: arch type ($d=0.35$), fishtail type ($d=1.5$), deltoid type ($d=d_*\simeq 2.5$) and bridge type curves of the three possible cases ($d=8$, $d=\mu^2{\rm e}^2$ and $d=20$, respectively). Here we look $\mathbb{H}^2(\rho)$ as the upper sheet of the hyperboloid model viewed from the top of the $x_3$-axis}
		\label{curvash21}
	\end{centering}
\end{figure}

\begin{figure}[hbtp]
	\begin{centering}{	\includegraphics[width=0.25\textwidth]{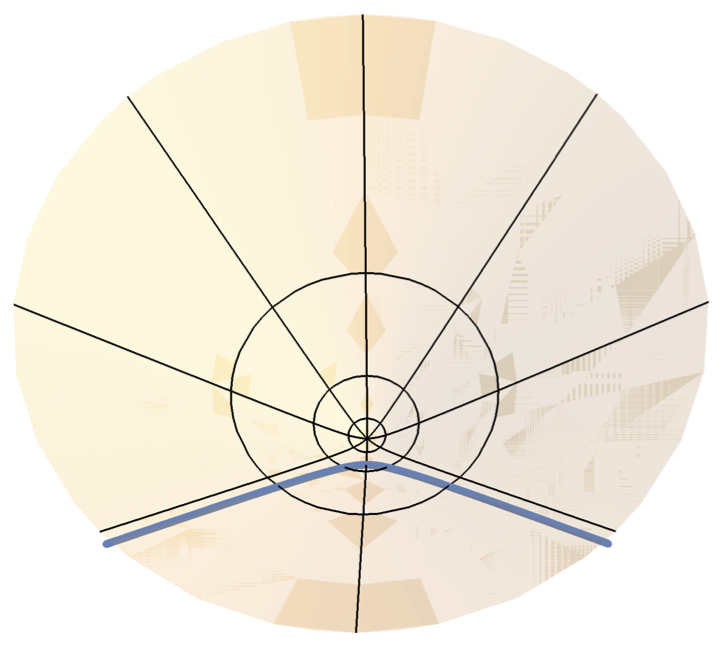}\quad  	\includegraphics[width=.25\textwidth]{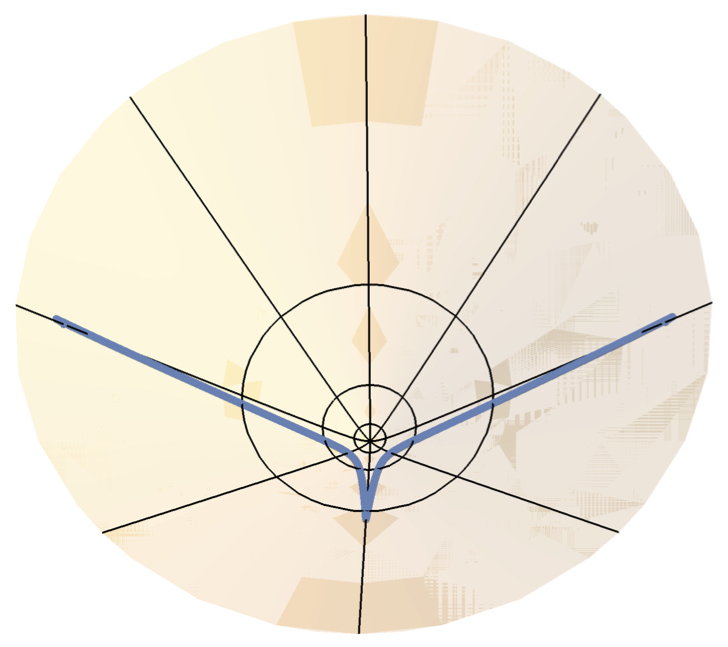}}
		\caption{Non-Lilienthal's type critical curves of $\mathbf{\Theta}_\mu$ in $\mathbb{H}^2(\rho)$. The first one is an hypercycle type curve for $d<\rho x_{-}^2\log x_{+}$ (here, $\rho=-1$, $\mu=1.25$ and $d=20$); while, the second one corresponds with $d>\rho x_{-}^2\log x_{+}$ and is an anchor type curve (here, $\rho=-1$, $\mu=1.25$ and $d=36$)}
		\label{curvash22}
	\end{centering}
\end{figure}

\section*{Acknowledgments} 
Rafael L\'opez was partially supported by MEC-FEDER grant no. MTM2017-89677-P. \'Alvaro P\'ampano was  partially supported by MINECO-FEDER grant MTM2014-54804-P, Gobierno Vasco grant IT1094-16 and by Programa Posdoctoral del Gobierno Vasco, 2018


\end{document}